\documentclass[11pt,letterpaper]{amsart}
\usepackage[svgnames]{xcolor}
\usepackage[colorlinks,allcolors={DarkBlue}]{hyperref}

\usepackage{comment}

\usepackage{esint}
\usepackage{amssymb}
\usepackage{tikz}
\usetikzlibrary{calc}
\usepackage{graphicx}
\usepackage{amsrefs,enumitem}

\usepackage{manfnt}
\usepackage{graphicx} % for \scalebox

\newtheorem*{theorem*}
{Theorem}
\newtheorem{theorem}{Theorem}
\newtheorem{lemma}[theorem]{Lemma}
\newtheorem{corollary}[theorem]{Corollary}
\newtheorem{proposition}[theorem]{Proposition}
\theoremstyle{definition}

\newtheorem{remark}[theorem]{Remark}

\newtheorem{definition}[theorem]{Definition}

\newcommand{\eref}[1]{(\ref{e.#1})}
\newcommand{\tref}[1]{Theorem \ref{t.#1}}
\newcommand{\lref}[1]{Lemma \ref{l.#1}}
\newcommand{\pref}[1]{Proposition \ref{p.#1}}
\newcommand{\cref}[1]{Corollary \ref{c.#1}}
\newcommand{\fref}[1]{Figure \ref{f.#1}}
\newcommand{\sref}[1]{Section \ref{s.#1}}

\numberwithin{theorem}{section}
\numberwithin{equation}{section}

\newcommand{\R}{\mathbb{R}}

\newcommand{\grad}{\nabla}

\def\Xint#1{\mathchoice
{\XXint\displaystyle\textstyle{#1}}%
{\XXint\textstyle\scriptstyle{#1}}%
{\XXint\scriptstyle\scriptscriptstyle{#1}}%
{\XXint\scriptscriptstyle\scriptscriptstyle{#1}}%
\!\int}
\def\XXint#1#2#3{{\setbox0=\hbox{$#1{#2#3}{\int}$ }
\vcenter{\hbox{$#2#3$ }}\kern-.6\wd0}}

\def\dashint{\Xint-}

\newcommand{\ep}{\varepsilon}
\newcommand{\e}{\varepsilon}
\newcommand{\osc}{\mathop{\textup{osc}}}

\definecolor{darkgreen}{rgb}{0,0.4,0}
\def\strikethrough#1{\setbox0\hbox{#1}\rlap{#1}\hbox to \wd0{\hss\strikebox\hss}}
\def\strikebox{\vrule height 0.6\ht0 depth -0.4\ht0 width 1.1\wd0}

\begin{document}

\title{Solutions of the Bernoulli one-phase problem with a defect}
\author{William M Feldman}
\address{Department of Mathematics, The University of Utah, Salt Lake City}
\author{Inwon C Kim}
\address{Department of Mathematics, University of California, Los Angeles}
\begin{abstract}
We study the far-field behavior of solutions of the one-phase Bernoulli free boundary problem in the exterior of a ball, and of entire solutions with a single compactly supported inhomogeneity of the free boundary condition, which we call a {\it defect}. For solutions which blow down to a half-plane solution {(proper solutions)} we establish an asymptotic expansion at infinity: in dimension $d \geq 3$ the free boundary height converges to a limit at rate $|x|^{2-d}$ with a capacity-type coefficient, while in dimension $d=2$ the expansion carries a logarithmic term. {A significant novelty is that the expansions are quantitative and uniform over all the proper solutions.}
\end{abstract}

\maketitle

\setcounter{tocdepth}{1}
\tableofcontents

\section{Introduction}

We consider nonnegative continuous viscosity solutions of the one-phase Bernoulli free boundary problem
\begin{equation}\label{e.bernoulli-intro}
    \begin{cases}
        \Delta u = 0 & \hbox{ in } \{u>0\} \cap U\\
        |\grad u| = Q(x) &\hbox{ on } \partial \{u>0\} \cap U,
    \end{cases}
\end{equation}
in two closely related settings: the \emph{exterior problem}, where $U = \R^d \setminus B_1$ and $Q \equiv 1$, and the \emph{single-site defect problem}, where $U = \R^d$ and the coefficient differs from $1$ only on a compact set. In the latter case a function  $q : \R^d \to (-1,\infty)$, $q\in C_0(\overline{B_1(0)})$, represents a single chemical defect on the surface, and the problem reads
\begin{equation}\label{e.single-site}
     \begin{cases}
         \Delta u = 0 & \hbox{in } \{u>0\}  \\
         |\grad u| = 1+q(x) &\hbox{on } \partial \{u>0\} .
     \end{cases}
\end{equation}
In this paper we provide a precise description of the behavior at infinity of solutions which are asymptotically planar, quantified in terms of the size of the defect. The far-field coefficient appearing in the expansion plays the role of a capacity. Unlike the classical capacities of elliptic exterior problems, it is not unique for a given $q$. Instead the free boundary problem admits a family of capacitory potentials. The structure of that family, and its consequences for the pinning of free boundaries by defects, are the subject of the companion paper \cite{FKdilute}. While our results here serve as the basis for the companion paper \cite{FKdilute}, they also make independent contributions to the Bernoulli free boundary problem in the exterior setting.

Our analysis is focused on \emph{proper solutions} of \eqref{e.single-site}:
\begin{definition}\label{d.proper}
 $u \in C(\R^d)$ is \emph{proper} with direction $e$ if the sequence $u_r(x) : = r^{-1}u(rx)$ converges to $(x \cdot e)_+$ as $r\to 0$ in the following sense:
\begin{equation}\label{e.s2-blow-down-cond}
(u_r(x) - x\cdot e){\bf 1}_{\{u_r>0\}} \to 0 \ \hbox{ locally uniformly in $\R^d$ as } r\to 0.
\end{equation}
\end{definition}

This rather strong definition is introduced to rule out solutions that blow down to singular profiles, chiefly the two-plane solution $|x \cdot e|$, a well-known singular profile in the Bernoulli problem \cites{CS,KriventsovWeiss,JerisonKamburov2}. We mention that it suffices to only consider proper solutions for our interest in the application to pinning intervals, see \cite{FKdilute}. 

\subsection{Main results}
Throughout the paper we normalize $e = e_d$ by rotation. Our first main result concerns classical solutions of the exterior problem which are uniformly close to planar in gradient. 

\begin{theorem}[\tref{flat-exterior-original-coord}]\label{t.main-flat-exterior}
    Let $u$ be a classical $C^1$ solution of the homogeneous problem ($Q \equiv 1$) in $\R^d \setminus B_1$ such that $u(x) - x_d$ is bounded from above or from below in $\overline{\{u>0\}} \setminus B_1$, and such that
    \[\sup_{\overline{\{u>0\}} \setminus B_1} |\grad u - e_d| \leq \eta_1 \quad \hbox{ for } \,\, 0 < \eta_1(d) <1/2.\]
    Then there are $s, k \in \R$ and $C(d) \geq 1$ so that, in $d \geq 3$,
        \[|x|^{d-1}\left|u(x) - (x_d+s  + k |x|^{2-d})\right| \leq  C\max_{(B_2 \setminus B_1) \cap \{u>0\}} (u(x) - x_d),\]
    while in $d = 2$,
        \[\left|u(x) -(x_d +  k \log |x|)\right| \leq  C\max_{(B_2 \setminus B_1) \cap \{u>0\}}|u(x) - x_d|.\]
   Furthermore $|k| \leq C \max_{(B_2 \setminus B_1) \cap \{u>0\}} |u(x) - x_d|$.
\end{theorem}

The gradient flatness hypothesis lets us transfer the entire analysis to the partial hodograph coordinates, where the free boundary problem becomes a nonlinear elliptic equation in a half-space with a nonlinear oblique (Neumann-type) boundary condition; the precise statement in those variables is \tref{flat-exterior}. The proof uses explicit homogeneous barriers, a quantitative linearization of the hodograph PDE, and a Kelvin transform argument reducing the expansion at infinity to a boundary regularity estimate at the origin. The proof is flexible and applies to any nonlinear oblique problem whose coefficients satisfy the structural bounds \eref{A-property} and \eref{N-property}.

Our second main result removes the global flatness hypothesis. It gives the expansion for arbitrary proper solutions of the defect problem, with constants quantified by the defect size.

\begin{theorem}[\tref{asymptotic-expansion-at-infty}]\label{t.main-expansion}
   Let $u$ be a proper solution of \eref{single-site} (with direction $e = e_d$).
   \begin{enumerate}[label = (\roman*)]
    \item If $d=2$ there is $ k(u)\in \R$ so that
    \[\sup_{x \in \overline{\{u>0\}}}|u(x) -(x_d  +k \log |x|)| < +\infty,\]
       and there are universal constants $\sigma_0 \in (0, 1/2]$ and $C\geq 1$ so that if $|q|\leq \sigma_0$ then $C\min q \leq k(u) \leq C \max q$. For general $q$ there is $C(\min q)$ so that $k(u) \geq -C$.
       \item If $d \geq 3$ there are $s(u), k(u) \in \R$ so that
   \[|u(x) -(x_d +  s -k|x|^{2-d})| \leq C_0(1+|x|)^{1-d} \ \hbox{ for } \ x \in \{u>0\},\]
   and if $k \neq 0$ then $|s|,|k| \leq C_0=C_0(d,\min q,\max q)$. If moreover $\max |q| \leq \sigma_0(d)$ then $C_0\leq C(d) \max |q|$.
   \end{enumerate}
\end{theorem}

The quantification is the difficult part: properness gives flatness only at some large non-universal scale, and \tref{flat-exterior-original-coord} alone would produce constants depending on the solution. To control the scale we construct explicit sub- and supersolution barriers for the cylindrical model problem \eref{barrier-eqn} with the correct far-field expansion \eref{barrier-expansion-3d}--\eref{barrier-expansion-2d}, see \pref{barrier-prop}. We perform a nearly-linear patching construction for small defects and two genuinely nonlinear constructions: a supersolution in all $d \geq 2$ and a subsolution in $d \geq 3$. This kind of asymmetry between the sub and supersolution conditions is typical in the Bernoulli one-phase problem.  Whether nonlinear-regime subsolutions exist in $d=2$ remains an interesting open question. 

\subsection{Comparison with previous literature}

The study of far field behavior of solutions of the single site problem \eqref{e.single-site} is related to the ``exterior problem" for Bernoulli free boundaries. Our usage of hodograph transform in order to study foliating families of solutions indexed via far-field expansions are reminiscent of some works on solutions near singular cones \cites{engelstein2025asymptotic,de2022inhomogeneous,EngelsteinSpolaorVelichkov}, see \cite{engelstein2025asymptotic} for more discussion of the literature including related literature in the minimal surface theory. There is significant literature studying the far-field behavior of exterior solutions for quasi-linear elliptic equations, especially for graphical minimal surfaces, see for example \cite{Bers,Simon,Schoen}.  A recent work \cite{FR}, appearing after the original version of this paper was posted in \cite{FKoriginal}, presents a nice new proof of the far field behavior via an improvement of flatness in annuli argument.  There are significant differences with our work. Most works on exterior problems consider variational solutions and have access to monotonicity formulae, whereas we work with viscosity solutions. We also need uniform, quantitative estimates, achieved with explicit barriers, which we carry out in Sections \ref{s.flat-asymptotics} and \ref{s.general-asymptotics}. In addition our two-dimensional results address logarithmic tail behavior, which appear to be new in the study of Bernoulli free boundaries.

\subsection{Outline of the paper} Section \ref{s.prelim-exterior} fixes the setting and collects preliminary results: the sliding comparison principle, blow-down analysis of one-sided flat exterior solutions including a removable singularity lemma and strong maximum principles, flat regularity theory, the partial hodograph transform, and Harnack inequalities in both coordinate systems. Section \ref{s.flat-asymptotics} proves \tref{flat-exterior} (and with it Theorem \ref{t.main-flat-exterior}) via barriers, quantitative linearization, and the Kelvin transform. Section \ref{s.general-asymptotics} constructs the barriers of \pref{barrier-prop} and proves \tref{asymptotic-expansion-at-infty}. 

\subsection{Acknowledgments}
W.F.'s research is partially supported by NSF DMS-2407235. Part of this work was completed during W.F.'s visit at the ESI, he thanks the ESI for hosting him. I.K.'s research is partially supported by NSF DMS 2452649. Part of this work was completed during I.K's visit at KIAS, and she thanks KIAS's hospitality.

\subsection{Declaration of AI use} The original version of this paper \cite{FKoriginal} was written without AI assistance. For this revised version we used AI tools only for editing and presentation ideas in the introduction.

    \section{Setting and preliminary results}\label{s.prelim-exterior}
    {In this section we introduce the settings we will work in throughout the paper. We will recall several results from the literature which will be useful later. }

    \subsection{Bernoulli free boundary problems and solution notions}

We will consider continuous viscosity solutions $u \in C(U)$, $u \geq 0$, of the Bernoulli free boundary problem in a domain $U \subset \R^d$ with a continuous and positive coefficient field $Q(x)$
\begin{equation}\label{e.bernoulli-basic-prelim}
     \begin{cases}
         \Delta u = 0 & \hbox{in } \{u>0\} \cap U \\
         |\grad u| = Q(x) &\hbox{on } \partial \{u>0\} \cap U. 
     \end{cases}
\end{equation}
In many cases we will be able to reduce to studying a homogeneous problem
\begin{equation}\label{e.bernoulli}
    \begin{cases}
        \Delta u = 0 & \hbox{ in } \{u>0\} \cap U\\
        |\grad u| = 1 &\hbox{ on } \partial \{u>0\} \cap U.
    \end{cases}
    \end{equation}
 Background on the definition and basic theory of viscosity solutions can be found, for example, in \cites{CS,feldman2021limit} and in \cite{FKdilute}*{Section 2}.  We will also sometimes consider classical solutions of \eref{bernoulli-basic-prelim}, $u$ is a classical solution if $u \in C^1(\overline{\{u>0\}} \cap U) \cap C^2(\{u>0\} \cap U)$.  Of course both viscosity and classical solutions are smooth $\{u>0\} \cap U$ since they are harmonic in that open set.

    \subsection{Sliding comparison} 
The standard comparison principle in bounded domains does not hold for solutions of \eqref{e.bernoulli-basic-prelim}. Pinning on heterogeneities is one source of such non-uniqueness, but non-uniqueness even occurs in the homogeneous problem \eref{bernoulli}. Instead, we will often use the following {sliding comparison principle}, where we compare a supersolution $u$  with a continuously varying family of regular subsolutions $v_t$ from below. We can also do similar if $u$ is a subsolution and $v_t$ is a continuously varying family of supersolutions.  Note that $u$ does not need to be regular, only the sliding family.

\begin{lemma} [Sliding comparison \cite{CS}*{Theorem 2.2}]\label{l.sliding-comparison}
Let $u \in C(\bar{U})$ be a viscosity supersolution of \eqref{e.bernoulli-basic-prelim}.  Let $v_0 \in C(\bar{U})$ satisfy the following:
\begin{enumerate}[label = (\roman*)]
	\item $v_t(x):=v_0(x+te_d)$ are classical subsolutions of \eqref{e.bernoulli-basic-prelim} for all $t \in [0,T]$;
    \item $v_0\leq u$ in $U$;
    \item $v_t \leq u$ on $\partial U$ and $v_t <u$ in $\overline{\{v (x+te_d)>0\}} \cap \partial U$ for $0\leq t\leq T$.
\end{enumerate}
    Then $v_t \leq u$ in $U$  for  $0\leq t\leq T$.
\end{lemma}

We will often apply the sliding comparison in the whole space by ensuring the boundary ordering property holds ``at $\infty$", that is on the boundary of all sufficiently large radius balls.

    \subsection{$C^{1,\alpha}$ regularity for flat solutions} General viscosity solutions of \eref{bernoulli} may not be regular. However, sufficiently flat solutions are indeed classical in a slightly smaller domain.

\begin{theorem}[Caffarelli]\label{t.flat-implies-c1alpha}
     For any $\alpha \in (0,1)$ there is $\eta_0(\alpha,d)>0$ and $C(d) \geq 1$ so that the following holds. If $u$ is a viscosity solution of \eref{bernoulli} in $B_1$ and
        \[ (x_d)_+ \leq u(x) \leq (x_d+\eta)_+  \ \hbox{ in } \ B_1  \ \hbox{ with } \ \eta \leq \eta_0
        \]
        then $u \in C^{1,\alpha}(\overline{\{u>0\}} \cap B_{1/2})$ and
        \[\sup_{\overline{\{u>0\}} \cap  B_{1/2}}|\grad u - e_d| \leq C\eta.\]
\end{theorem}

For exterior solutions which blow down to a half-planar solution we can apply this regularity theory sufficiently large annuli.

\begin{lemma}\label{l.blow-down-grad}
    Suppose that $u$ solves \eref{bernoulli} in $\R^d \setminus B_1$ and blows down to $(x_d)_+$ as in \eref{s2-blow-down-cond}. Then there is $R_0>0$ depending on $u$ so that for any $\alpha \in (0,1)$
    \[u \in C^{1,\alpha}(\overline{\{u>0\}} \setminus B_{R_0})\]
    and for all $r \geq R_0$
     \[\sup_{\overline{\{u>0\}} \cap \partial B_r}|\grad u - e| \leq C\sup_{B_{2} \setminus B_{1/2}} |r^{-1}u(rx) - (x\cdot e)_+|.\]
\end{lemma}
  
 \subsection{Partial hodograph transform }\label{s.hodograph}

Now we recall the partial hodograph transform. This transformation was first introduced as a tool to prove higher regularity of free boundaries by \cite{kinderlehrer1977regularity}. While requiring a $C^1$, planar-like solution as a starting point, this transformation has served as one of the main tools to obtain higher regularity of the free boundary. For us the transformation puts us in a PDE setting where we can apply classical higher regularity, Harnack, and Kelvin transform techniques to study the exterior asymptotics. It is also convenient for the construction of barriers, especially for $d=2$ case, where the logarithmic far-field growth of the free boundary makes it challenging to construct barriers in original coordinates.

  Let $u\in C^2(\{u>0\}) \cap C^1(\overline{\{u>0\}})$ be a classical solution of \eqref{e.bernoulli} in an open neighborhood $U$ of $x_0\in \partial\{u>0\}$.  We assume that $\partial_{x_d} u>0$  in $\overline{\{u>0\}} \cap U$.  We now define the new coordinates $y=(y', y_d)$ by 
\begin{equation}\label{hodo.t} 
y' := x', \ y_d = u(x),  \ \hbox{ and } \  v(y) := x_d -y_d.
\end{equation} 
Under our hypotheses on $u$ the coordinate transform defines a diffeomorphism of $\overline{\{u>0\}} \cap U$ onto its image, a set $\mathcal{N} \cap \{y_d \geq 0\}$.

We now derive the $y$-cordinate  PDE  in the domain $\{y_d\geq 0\}\cap\mathcal{N}$. Observe that 
$$
\nabla_{y'} y_d =0= \nabla_{y'} u + \partial_d u \nabla_{y'} x_d,
$$
and so
$$
\nabla_{y'} v = \nabla_{y'} x_d = (-\nabla_{x'} u) (\partial_d  u )^{-1}\hbox{
and }
\partial_{y_d} v = \partial_{y_d} x_d -1 = (\partial_d u) ^{-1}-1.
$$
  We also compute
$$
\sqrt{1+|\nabla'_y v|^2} = |Du|(\partial_d u)^{-1} = |Du|(1+\partial_{y_d} v) 
\,\,\hbox{ and }\,\, 
(\nabla v)'/(1+(\nabla v)_d) = -\nabla_{x'}u.
$$

Thus, the PDE in the Hodograph coordinates is:
\begin{equation}\label{e.hodo-PDE}
\begin{cases}
    \textup{tr}(A(\grad_y v)D^2_yv) =0 & \hbox{ in } \ \{y_d>0\} \cap \mathcal{N}; \\
    \partial_{y_d}v =N(\nabla'_y v)  :=\sqrt{1+|\grad_y 'v|^2}-1  & \hbox{ on } \partial \{y_d>0\} \cap \mathcal{N},
\end{cases}
\end{equation}
 where
\begin{equation}\label{e.hodograph-PDE-A-formula}
    A(p) = \left[\begin{array}{cc}
    \textup{I}_{d-1} & (1+p_d)^{-1}p' \\
    (1+p_d)^{-1}(p')^T & \frac{1+|p'|^2}{(1+p_d)^2}
\end{array}\right].
\end{equation}
\begin{remark}
    Note that $A(p)$ is elliptic if $|p|<1$ and satisfies
        \begin{equation}\label{e.A-property}
        \|A(p)-I\| \leq C|p| \ \hbox{ for } \ |p| \leq 1/2
    \end{equation}
    and
       \begin{equation}\label{e.N-property}
        |N(p')| \leq C|p'|^2 \ \hbox{ for } \ |p'| \ll 1.
    \end{equation}
    These are the main properties of $A$ and $N$ that we will use later.
\end{remark}

    \subsection{Regularity and Harnack for nonlinear oblique boundary value problems}

    We use a Harnack inequality and higher regularity estimates for the nonlinear Neumann problem \eref{hodo-PDE}. In the contexts where we use these results, our solutions $v$ of \eref{hodo-PDE} will be at least $C^{1,\alpha}$ with $|\grad v| \leq \eta \ll 1$. Therefore we can view \eref{hodo-PDE} as a linear uniformly elliptic PDE with a linear uniformly oblique boundary condition with measurable coefficients.

    Applying Lieberman's Harnack inequality \cite{Lieberman}*{Theorem 3.3} for linear uniformly elliptic PDE with uniformly oblique boundary condition and measurable coefficients, we arrive at the following Harnack inequality for the Hodograph PDE \eref{hodo-PDE}.
    \begin{theorem}[Corollary of Lieberman's Harnack inequality \cite{Lieberman}*{Theorem 3.3}]\label{t.hodograph-harnack}
    For any $\eta < 1$ there is a constant $C(d,\eta)>1$ such that for any non-negative $C^{1}$ solution of \eref{hodo-PDE} in $B_{1}^+ = B_1 \cap \{y_d>0\}$ called $v$ with $|\grad v|\leq \eta < 1$,
    \[\sup_{B_{1/2}^+} v \leq C \inf_{B_{1/2}^+} v.\]  
    \end{theorem}
    
We also need elliptic regularity type estimates up to second order with the correct scaling in large balls. The qualitative $C^\infty$ regularity of $C^{1,\alpha}$ solutions of \eref{hodo-PDE} was proved in the original paper applying hodograph techniques by Kinderlehrer and Nirenberg \cite{kinderlehrer1977regularity}*{p. 386}.  Then we can apply Lieberman and Trudinger's \cite{LiebermanTrudinger} a-priori estimates for $C^2$ solutions of nonlinear uniformly elliptic problems with nonlinear oblique boundary conditions.  
\begin{theorem}[See \cite{LiebermanTrudinger}*{Theorem 1.1}]\label{t.hododgraph-regularity}
    If $v$ is a $C^{2}$ solution of \eref{hodo-PDE} in $B_r^+$ with $|\grad v| \leq \eta < 1$ then
    \[r^k|\grad^k v(0)| \leq C\osc_{B_r^+} v \ \hbox{ for } \ k = 1,2.\]
\end{theorem}
Note that the nonlinear problem \eref{hodo-PDE} is invariant under hyperbolic rescaling $v \mapsto rv(\cdot/r)$, which is how we are applying \cite{LiebermanTrudinger}*{Theorem 1.1}.

\subsection{Harnack inequality for the Bernoulli problem}Sometimes it is convenient to have the Harnack inequality \tref{hodograph-harnack} directly available in the original coordinates.  By combining \tref{flat-implies-c1alpha} with the hodograph transform and \tref{hodograph-harnack} one can derive the following Harnack inequality for flat solutions of \eref{bernoulli}. 
    \begin{corollary}\label{c.harnack-bernoulli}
         Let $u$ solve \eref{bernoulli} in $B_1$. Then there is $\eta_0>0$ and $C \geq 1$ depending on dimension so that the following holds. If 
        \[ (x_d)_+ \leq u(x) \leq (x_d+\eta_0)_+  \ \hbox{ in } \ B_1 
        \]
        then
        \[\sup_{\{u>0\} \cap B_{1/2}} (u(x) - x_d) \leq C \inf_{\{u>0\} \cap B_{1/2}} (u(x) - x_d).\]
        Similarly if
        \[ (x_d-\eta_0)_+ \leq u(x) \leq (x_d)_+  \ \hbox{ in } \ B_1 
        \]
        then
        \[\sup_{\{u>0\} \cap B_{1/2}} (x_d - u(x)) \leq C \inf_{\{u>0\} \cap B_{1/2}} (x_d - u(x)).\]
    \end{corollary}
    We do not directly use this result in the present paper, but we record it here because it is a useful statement which we have already made use of in the companion paper \cite{FKdilute}.

\section{Close to planar exterior solutions}\label{s.flat-asymptotics}

In this section we analyze the asymptotic expansion of one-sided flat exterior solutions, i.e. of solutions to \eqref{e.bernoulli} in $U:= \R^d\setminus B_1$ with the property $|\grad u - e_d| \ll 1$ in $\{u>0\}$.

The following result will be applied in Section~\ref{s.general-asymptotics}  to our original solutions in the region that are away from the defects.

\begin{theorem}\label{t.flat-exterior-original-coord}
    Let $u$ be a classical $C^1$ solution of \eref{bernoulli} in $\R^d \setminus B_1$ such that $u(x) - x_d$ is bounded from above or from below in $\overline{\{u>0\}} \setminus B_1$.  If in addition
    \[\sup_{\overline{\{u>0\}} \setminus B_1} |\grad u - e_d| \leq \eta_1 \quad \hbox{ for } \,\, 0 < \eta_1(d) <1/2, \]  
    then the following holds.
    \begin{enumerate}[label = (\roman*)]
        \item For $ d \geq 3$, there is $C(d) \geq 1$ and $s, k \in \R$ so that
        \[\left|u(x) - (x_d+s  + k |x|^{2-d})\right| \leq  C|x|^{1-d}\osc_{(B_2 \setminus B_1) \cap \{u>0\}} (u(x) - x_d)\]
        and 
        \[ |k| \leq C\osc_{(B_2 \setminus B_1) \cap \{u>0\}}(u(x) - x_d).\]
        \item For  $d = 2$, there is $C \geq 1$ and $k \in \R$ so that
        \[\left|u(x) -(x_d +  k \log |x|)\right| \leq  C\max_{(B_2 \setminus B_1) \cap \{u>0\}}|u(x) - x_d|\]
        and
        \[ |k| \leq C\sup_{(B_2 \setminus B_1) \cap \{u>0\}}|u(x) - x_d|.\]
    \end{enumerate}
\end{theorem}

  Due to the small gradient condition, we can perform our analysis entirely in the {hodograph coordinates} (see \sref{hodograph}), which transforms our free boundary problem to a nonlinear elliptic problem in a half-space with a nonlinear Neumann condition.

\subsection{One-sided flat solutions in hodograph variables}

Now we state a version of \tref{flat-exterior-original-coord} in the hodograph variable.

\begin{theorem}\label{t.flat-exterior}
  Let $v$ be a smooth and one-sided bounded solution of \eqref{e.hodo-PDE} with $\mathcal{N}=\R^d\setminus B_1$. There is $\eta_0(d)\in (0,1/2)$ such that if  $\sup_{\R^d_+ \setminus B_1} |\grad v| \leq \eta_0$ then the following holds.
    \begin{enumerate}
        \item For $ d \geq 3$, there  is $C(d) \geq 1$ and $s, k \in \R$ so that
        \[\left|v(y) - s  - k |y|^{2-d}\right| \leq  C|y|^{1-d}\osc_{(B_2 \setminus B_1)^+} v\]
        and 
        \[|s| \leq \dashint_{\partial B_1 \cap \R^d_+} v +  C\osc_{(B_2 \setminus B_1)^+} v \ \hbox{ and } \ |k| \leq C\osc_{(B_2 \setminus B_1)^+}v.\]
        \item For $d = 2$, there is $C \geq 1$ and $k \in \R$ so that
        \[\left|v(y) -  k \log |y|\right| \leq  C\max_{(B_2 \setminus B_1)^+} |v|\]
        and
        \[ |k| \leq C \max_{(B_2 \setminus B_1)^+} |v|.\]
    \end{enumerate}
\end{theorem}
The remainder of \sref{flat-asymptotics} will be occupied with the proof of \tref{flat-exterior}.

\begin{remark}
 Our analysis can be applied to any nonlinear Neumann problem of the type \eref{hodo-PDE} with operators $A(p)$ and $N(p')$ satisfying the estimates \eqref{e.A-property} and  \eref{N-property}.

\end{remark}

\subsection{Initial barriers} First we establish the existence of smooth homogeneous super and subsolution barriers, to be used in this section. 

In dimension $d=2$ it is very convenient to have barriers with the correct logarithmic behavior at highest order. We show the existence of such barriers in the next result. 

\begin{lemma}\label{l.logarithmic-barriers}
    Assume that $A$ and $N$ satisfy \eref{A-property} and \eref{N-property} and $d=2$. Define the barriers
    \[\psi_\pm(x) := \log |x|\pm\log(1+\log |x|)\pm \frac{x_d}{|x|^2}.\]
    There is $\varsigma_0>0$ sufficiently small so that if $0 \leq \varsigma \leq \varsigma_0$ then $\varsigma \psi_+$ is a subsolution and $\varsigma \psi_-$ is a supersolution of \eref{hodo-PDE} in $\R^2_+ \setminus B_{1}$. 
\end{lemma}

Note that, although the actual hodograph PDE is not invariant under negation $v \mapsto -v$, the properties \eref{A-property} and \eref{N-property} are. So, for example, we can also conclude under the hypotheses of \lref{logarithmic-barriers} that $-\varsigma\psi_+$ is a supersolution of \eref{hodo-PDE} in $\R^2_+ \setminus B_{1}$.

\begin{proof}
We check the solution properties by direct computation.  Note that $\log |x|$ is harmonic and its derivatives $D^k\log|x|$ are homogeneous of order $k$. Also $\frac{x_d}{|x|^2}$, which is $\partial_d (\log |x|)$, is harmonic, and its $k$th derivatives are homogeneous of order $k+1$.

We record the derivative of $\log(1+\log |x|)$
\[\grad \log(1+\log|x|) = \frac{1}{1+\log |x|} \frac{x}{|x|^2}\]
and
\[D^2\log(1+\log |x|) = \frac{1}{|x|^2(1+\log |x|)}(I - 2\frac{x \otimes x}{|x|^2})+\frac{1}{|x|^2(1+\log |x|)^2} \frac{x \otimes x}{|x|^2}.\]

Since the trace of $(I - 2\frac{x \otimes x}{|x|^2})$ vanishes in dimension $d=2$ we find that
    \begin{align*}
        \textup{tr}(A(\varsigma\grad\psi_\pm) D^2\psi_\pm) &=\textup{tr}(D^2 \psi_\pm) +\textup{tr}((A(\varsigma\grad \psi_\pm) - I)D^2\psi_\pm)\\
        &=\pm\frac{1}{|x|^2(1+\log|x|)^2} + \textup{tr}((A(\varsigma\grad \psi_\pm) - I)D^2\psi) \\
        &=\pm\frac{1}{|x|^2(1+\log|x|)^2} + O(\varsigma\frac{1}{|x|^3})
    \end{align*}
    For the last equality we used \eref{A-property} and the homogeneous upper bounds, which hold in $\R^d \setminus B_1$, $|\grad \psi_\pm| \leq C|x|^{-1}$ and $|D^2\psi_\pm| \leq C|x|^{-2}$.  Thus, for $|\varsigma| \leq \varsigma_0$ sufficiently small
    \[\pm\textup{tr}(A(\varsigma\grad \psi_\pm) D^2 \psi_\pm) \geq 0 \ \hbox{ for } \ |x| \geq 1.\]
    
     Next we check the subsolution properties on $\partial \R^d_+$. Noting the formula
   \begin{equation}\label{e.psiout-grad-formula}
       \grad \psi_\pm(x) =  \frac{x}{|x|^2}\left(1+\frac{1}{1+\log|x|}\right)\pm \frac{1}{|x|^4}\left(e_d|x|^2-2x_dx\right)
   \end{equation}
     and so
     \[\partial_{x_d} \psi_\pm(x) = \pm\frac{1}{|x|^2} \ \hbox{ on } \ x_d = 0.\]
     On the other hand, using again the homogeneous upper bound $|\grad'\psi_\pm| \leq C|x|^{-1}$ in $|x| \geq 1$ and \eref{N-property}
     \[N(\varsigma\grad'\psi_\pm) \leq \varsigma^2\frac{1}{|x|^2}.\]
     So, again for $\varsigma>0$ sufficiently small,
     \[\varsigma\partial_{x_d} \psi_+(x) \geq N(\varsigma\grad'\psi_+) \ \hbox{ and } \ \varsigma\partial_{x_d} \psi_-(x) \leq N(\varsigma\grad'\psi_-).\]
\end{proof}

The barriers for $d\geq 3$ are given below. We omit the proof, since it is standard (and easier than the case $d=2$), based on the fact that $A(p)$ has ellipticity ratio close to $1$ when $|p|$ is sufficiently small.  

\begin{lemma}\label{l.fundie-barriers}
    For $d\geq 3$ and any $0 < \delta < d-1$ there is $c_\delta=c_{\delta}(d)>0$ so that, calling $\sigma = \textup{sgn}(d-2-\delta)$,
     \[\phi_+(y)= \sigma c_\delta |y|^{2-d+\delta} \ \hbox{ and } \ \phi_-(y) = -\sigma c_\delta |y+\frac{1}{2}e_d|^{2-d+\delta} \]
 are respectively a supersolution and subsolution of \eref{hodo-PDE}. 
\end{lemma}

\subsection{A growth bound via the homogeneous barriers} Next we show a barrier argument which, in a certain sense, controls the growth of $v$ on a large annulus $(B_r \setminus B_1)^+$ in terms of its growth on $(B_2 \setminus B_1)^+$.
  
\begin{lemma}\label{l.one-sided-bound-2}
    Suppose that $d \geq 3$, $r \geq 2$, and $v$ solves \eref{hodo-PDE} in $(B_r \setminus B_1)^+$. There is $\eta_0(d)>0$ sufficiently small so that if $\sup_{\R^d_+ \setminus B_1} |\grad v| \leq \eta_0(d)$ then
    \[ \max_{\partial B_r} v \geq \max_{\partial B_1} v - C(\max_{\partial B_1} v -\max_{\partial B_2} v)_+\]
    and
    \[\min_{\partial B_r} v  \leq\min_{\partial B_1} v+ C(\min_{\partial B_2} v -\min_{\partial B_1} v)_+\]
    for a universal $C \geq 1$. 
\end{lemma}

Similar estimates are true of harmonic functions in annuli. Note that if $v$ were a harmonic polynomial the conclusions would be trivial, $\max_{\partial B_r} v \geq 0$ and $\min_{\partial B_r}v \leq 0$.

The proof follows a similar idea to \cite{ArmstrongSirakovSmart}*{Lemma 5.7}.

\begin{proof}
    
We will just argue for the lower bound on $\max_{\partial B_r} v$. The upper bound on the minimum is similar. Define
\[m(r) := \max_{\partial B_1} v - \max_{\partial B_r} v. \]
The goal is to show that $m(r) \leq Cm(2)_+$.

     If $m(r) \leq 0$ we are done. If $m(r) \geq 0$ then maximum principle in $(B_r \setminus B_1) \cap \R^d_+.$ implies that
\[ v(x) \leq \max_{\partial B_1} v \ \hbox{ in } \ (B_r \setminus B_1) \cap \R^d_+.\]
Implying that $m(2) \geq 0$ as well. So we have reduced to the case that $m(r) \geq 0$ and $m(2) \geq 0$.

    By \lref{fundie-barriers} the function $\phi(y) = \alpha|y|^{-1/2}$ is a supersolution of \eref{hodo-PDE} in $\R^d_+ \setminus B_1$ whenever $0 \leq \alpha \leq c_*(d)$. Define a barrier, with the non-negative (since $m(r) \geq 0$) constant $\alpha = \min \{ m(r),c_*\}$,
    \[\psi(y) := \max_{\partial B_1^+} v + \alpha(|y|^{-1/2}-1).\]
    By \lref{fundie-barriers} we have that $\psi$ is a supersolution of \eref{hodo-PDE} in $\R^d_+ \setminus B_{1}$ since $0 \leq \alpha \leq c_{*}$. Also 
    \[\psi(x) = \max_{\partial B_1^+} v \geq v (x) \ \hbox{ on }  \partial B_{1}.\]
  Note that, since $\alpha \leq m(r)$,  on $y \in\partial B_r^+$ we have
    \[\psi(y) = \max_{\partial B_1} v+\alpha(r^{-1/2}-1) \geq \max_{\partial B_1} v - \alpha \geq \max_{\partial B_r} v.\]
   Thus $\psi\geq v$ on $\partial B^+_r$. By comparison principle $\psi \geq v$ on $(B_r \setminus B_1)^+$. 
    
    Evaluating $\psi$ on $\partial B_{2}^+$ with the fact $\psi \geq v$, we conclude that
    \[(1-2^{-1/2})\min\{m(r),c_*\}\leq \max_{\partial B_1} v - v(y) \ \hbox{ for any } y \in \partial B_2^+\]
    and so
    \[\min\{m(r),c_*\}\leq (1-2^{-1/2})^{-1}(\max_{\partial B_1} v - \max_{\partial B_2} v) \leq (1-2^{-1/2})^{-1} c_d\eta_0 \leq \frac{1}{2}c_{*}\]
    as long as we choose $\eta_0$ sufficiently small.  Since $c_{*} > \frac{1}{2}c_{*}$ so we must have $\min\{m(r),c_{*}\} = m(r)$, and so,
    \[m(r) \leq (1-2^{-1/2})^{-1}m(2). \]

\end{proof}
By a very similar argument we can derive a growth bound in dimension $d=2$.

\begin{lemma}\label{l.max-max-bd-2d}
   Suppose that $d =2 $, $r \geq 2$, and $v$ solves \eref{hodo-PDE} in $(B_r \setminus B_1)^+$. There is $\eta_0>0$ sufficiently small so that if $\sup_{\R^d_+ \setminus B_1} |\grad v| \leq \eta_0$ then
    \[ \max_{\partial B_r} v \geq \max_{\partial B_1} v - C(1+\log r)(\max_{\partial B_1} v -\max_{\partial B_2} v)_+\]
    and
    \[\min_{\partial B_r} v  \leq\min_{\partial B_1} v+ C(1+\log r)(\min_{\partial B_2} v -\min_{\partial B_1} v)_+\]
    for some $C \geq 1$ universal.
\end{lemma}
\begin{proof}

We will just argue for the lower bound on $\max_{\partial B_r} v$. The upper bound on the minimum is similar. Define
\[m(r) := \max_{\partial B_1} v - \max_{\partial B_r} v. \]
The goal is to show that $m(r) \leq Cr^\delta m(2)_+$.

     If $m(r) \leq 0$ we are done. If $m(r) \geq 0$ then maximum principle in $(B_r \setminus B_1) \cap \R^d_+.$ implies that
\[ v(x) \leq \max_{\partial B_1} v \ \hbox{ in } \ (B_r \setminus B_1) \cap \R^d_+.\]
Implying that $m(2) \geq 0$ as well. So we have reduced to the case that $m(r) \geq 0$ and $m(2) \geq 0$.

    By \lref{fundie-barriers} the function $\phi(y) = -\alpha|y|^{\delta}$ is a supersolution of \eref{hodo-PDE} in $\R^d_+ \setminus B_1$ whenever $0 \leq \alpha \leq c_\delta$. Define a barrier, with the non-negative (since $m(r) \geq 0$) constant $\alpha = \min \{ r^{-\delta}m(r),c_\delta\}$,
    \[\psi(y) := \max_{\partial B_1^+} v - \alpha(|y|^{\delta}-1).\]
    By \lref{fundie-barriers} we have that $\psi$ is a supersolution of \eref{hodo-PDE} in $\R^d_+ \setminus B_{1}$ since $0 \leq \alpha \leq c_{\delta}$. Also 
    \[\psi(x) = \max_{\partial B_1^+} v \geq v (x) \ \hbox{ on }  \partial B_{1}.\]
  Note that, since $\alpha \leq r^{-\delta}m(r)$,  on $y \in \partial B_r^+$ we have
    \[\psi(y) = \max_{\partial B_1} v-\alpha(r^{\delta}-1) \geq \max_{\partial B_1} v - m(r) \geq \max_{\partial B_r} v.\]
   Thus $\psi\geq v$ on $\partial B^+_r$. By comparison principle $\psi \geq v$ on $(B_r \setminus B_1)^+$.

    Evaluating $\psi$ on $\partial B_{2}^+$ with the fact $\psi \geq v$, we conclude that
    \[(2^\delta - 1)\min\{r^{-\delta}m(r),c_\delta\}\leq \max_{\partial B_1} v - v(y) \ \hbox{ for any } y \in \partial B_2^+\]
    and so
    \[\min\{r^{-\delta}m(r),c_\delta\}\leq (2^\delta -1)^{-1}(\max_{\partial B_1} v - \max_{\partial B_2} v) \leq (2^\delta -1)^{-1} \eta_0 \leq \frac{1}{2}c_{\delta}\]
    as long as we choose $0<\eta_0 \leq c c_\delta(d)$ sufficiently small.  Since $c_{\delta} > \frac{1}{2}c_{\delta}$ so we must have $\min\{r^{-\delta}m(r),c_{\delta}\} = r^{-\delta}m(r)$, and so,
    \[m(r) \leq (2^\delta-1)^{-1}r^\delta m(2). \]

\end{proof}

\subsection{Bounds and limit at infinity}

First we consider bounded solutions, and show that bounded solutions have a limit at $\infty$. 

\begin{lemma}\label{l.height-exists}
    Suppose that $d \geq 2$ and $v$ solves \eref{hodo-PDE} in $\R^d_+ \setminus B_1$, $\sup_{\R^d_+ \setminus B_1} |\grad v| \leq \frac{1}{2}$,  and either $\liminf_{|x| \to \infty } v $ or $\limsup_{|x| \to \infty} v $ is finite. Then $\lim_{|x| \to \infty} v(x)$ exists.
\end{lemma}
 Note that if $v$ is bounded then the limit hypothesis is satisfied. On the other hand, \lref{height-exists} shows that if either $\liminf_{|x| \to \infty } v > -\infty$ or $\limsup_{|x| \to \infty} v < + \infty$ then $v$ is bounded.

 Also note that in the bound $|\grad v| \leq 1/2$ above the $1/2$ is not critical, it can be replaced with any $0<c<1$ for uniform ellipticity of the hodograph PDE.

\begin{proof}
    Similar to \cite{ArmstrongSirakovSmart}*{Lemma 5.8}.  Let's consider the case $\liminf_{|x| \to \infty} v(x) > -\infty$, the other case is similar.

    We may assume that $ \liminf_{|x| \to \infty} v(x)  = 0$. Let $\ep>0$. There is $R \geq 1$ sufficiently large so that $v(x) \geq -\ep$ for $|x| \geq R$. Also is a sequence $x_k \to \infty$ with $|x_k| \geq R$ so that $v(x_k) \leq \ep$, and call $r_k = |x_k|$. Since $v+\ep$ is a non-negative solution of \eref{hodo-PDE} in $\R^d_+ \setminus B_R$, by Harnack inequality \tref{hodograph-harnack}, 
    \[\sup_{\partial B_{r_k} \cap \R^d_+} (v+\ep) \leq C \inf_{\partial B_{r_k} \cap \R^d_+} (v+\ep) \leq 2C\ep\]
    or
    \[\sup_{\partial B_{r_k} \cap \R^d_+} v \leq (2C+1)\ep.\]
    By maximum principle in each $(B_{r_{k+1}} \setminus B_{r_k}) \cap \R^d_+$ also
    \[\sup_{R^d_+ \setminus B_{r_1}} v = \sup_k\sup_{(B_{r_{k+1}} \setminus B_{r_k}) \cap \R^d_+}v \leq (2C+1)\ep\]
    and so
    \[\limsup_{|x|\to\infty} v(x) \leq (2C+1)\ep.\]
    Since $\ep>0$ was arbitrary the limit supremum is $0$ agreeing with the limit infimum.
\end{proof}
Next we consider bounded solutions, which have a limit at $\infty$ due to \lref{height-exists}, and make the upper and lower bounds explicit depending on the limit at $\infty$.  Since we will refer to the limit often we define
\begin{equation}
    v_\infty := \lim_{|x| \to \infty} v(x).
\end{equation}
\begin{lemma}\label{l.range-of-exterior-v}
    Suppose that $d \geq 2$ and $v$ solves \eref{hodo-PDE} in $\R^d_+ \setminus B_1$, $\sup_{\R^d_+ \setminus B_1} |\grad v| \leq \frac{1}{2}$ and $v$ is bounded, and so $v_\infty$ exists. Then
    \[\min\{\min_{\partial B_1^+} v, v_\infty\} \leq v(x) \leq \max\{\max_{\partial B_1^+} v, v_\infty\} \ \hbox{ in } \ \R^d_+ \setminus B_1. \]
\end{lemma}
In fact the stated bounds on $v(x)$ are exactly the range of $v$ on $\R^d_+ \setminus B_1$. 
\begin{proof}
    Call $M = \max\{\max_{(\partial B_1)^+} v, v_\infty\}$, then the constant function $M+\delta$ is a solution of \eref{hodo-PDE} which is above $v$ on $\partial B_1$ and on all sufficiently large spheres $\partial B_R$.  Comparison principle in $(B_R \setminus B_1)^+$ implies that $v(x) \leq M+\delta$ in $\R^d_+ \setminus B_1$. Since $\delta$ is arbitrary $v(x) \leq M$.  The lower bound is proved in a symmetric way.
\end{proof}

\subsubsection{One-sided bounded solutions in $d \geq 3$} Now we consider solutions which are bounded only from one side. First, in dimension $d \geq 3$, we show that such solutions are bounded and control the limit $v_\infty = \lim_{|x| \to \infty} v$ in terms of the values in an annulus $B_2 \setminus B_1$. 

\begin{lemma}\label{l.one-sided-bound}
    Suppose that $d \geq 3$ and $v$ solves \eref{hodo-PDE}. There is $\eta_0(d)>0$ sufficiently small so that if $\sup_{\R^d_+ \setminus B_1} |\grad v| \leq \eta_0(d)$ and $v$ is bounded either from above or from below, then $v$ is bounded and
    \[ \max_{\partial B_1} v - C(\max_{\partial B_1} v -\max_{\partial B_2} v)_+ \leq v_\infty \leq \min_{\partial B_1} v + C(\min_{\partial B_2} v -\min_{\partial B_1} v)_+\]
    where $C \geq 1$ is universal.
\end{lemma}

\begin{proof}
    First we show that $v$ is bounded. Assume that $v$ is bounded from below, the other case can be argued similarly. Without loss assume that $v \geq 0$. By Harnack inequality, \tref{hodograph-harnack}, for all $r \geq 2$,
    \[\max_{\partial B_r^+} v \leq C\min_{\partial B_r^+} v.\]
    By \lref{one-sided-bound-2} 
    \[\min_{\partial B_r^+} v \leq \min_{\partial B_1} v+ C(\min_{\partial B_2} v -\min_{\partial B_1} v)_+,\]
    and so $\max_{\partial B_r^+} v$ is bounded independent of $r$.  Thus $v$ is bounded.

    Now we show that $v$ bounded implies the quantitative bounds in the statement. Since $v$ is bounded \lref{height-exists} implies that $v_\infty = \lim_{|x| \to \infty} v$ exists and \lref{range-of-exterior-v} implies that
   \[\min\{\min_{\partial B_1^+} v, v_\infty\} \leq v(x) \leq \max\{\max_{\partial B_1^+} v, v_\infty\} \ \hbox{ in } \ \R^d_+ \setminus B_1.\]
  In the case $v_\infty \in [\min_{\partial B_1^+} v,\max_{\partial B_1^+} v]$ the result is immediate.  
  
  Let's consider the case $v_\infty <\min_{(\partial B_1)^+} v $. In this case the claimed upper bound inequality is trivial, so we aim for the lower bound. Note that $v(x) - v_\infty$ is non-negative. Applying \lref{one-sided-bound-2} to $v(x) - v_\infty$ we find
  \[0 = \lim_{r \to \infty}\max_{\partial B_r^+} (v(x) - v_\infty) \geq \max_{\partial B_1} (v(x) - v_\infty) - C(\max_{\partial B_1} v -\max_{\partial B_2} v)_+.\]
  Rearranging this gives the desired lower bound on $v_\infty$. 
  
  The case $v_\infty > \max_{\partial B_1} v$ is argued symmetrically.
   
\end{proof}
Next we show a quantitative convergence rate to the exterior limit. This will follow from a barrier argument using the explicit barriers from \lref{fundie-barriers}, we will make use of the qualitative limit information established before to deal with the ``boundary at infinity" in the comparison argument.
\begin{lemma}\label{l.decay-est-d3}
      Suppose $d \geq 3$. For any $\delta>0$ there is $\eta_0(\delta,d)>0$ small so that: if $v$ solves \eref{hodo-PDE} is bounded either from below or from above and $\sup_{\R^d_+ \setminus B_1} |\grad v| \leq \eta_0$ then
    \[|y|^{d-2-\delta}| v(y) - v_\infty| \leq C\osc_{(B_2 \setminus B_1)^+} v  \ \hbox{ in } \ \R^d \setminus B_{1}.\]
\end{lemma}
\begin{proof}
    By \lref{one-sided-bound} and \lref{height-exists} the limit $v_\infty:= \lim_{|x| \to \infty} v(x)$ exists. Then \lref{one-sided-bound} implies that
\[\sup_{\partial B_1^+} |v-v_\infty| \leq C\osc_{(B_2 \setminus B_1)^+} v \leq C\eta_0.\]
    Let $\eta_0$ sufficiently small so that $C\eta_0 \leq c_\delta$ from \lref{fundie-barriers}. Then, by \lref{fundie-barriers},
    \[\phi(x):= (\sup_{\partial B_1^+} |v-v_\infty|) |y|^{2-d+\delta}\]
    is a supersolution of \eref{hodo-PDE} in $\R^d_+ \setminus B_{1}$.
    Let $\ep>0$ arbitrary and $R_1 \geq 1$ sufficiently large so that
    \[v(y) \leq \ep \ \hbox{ on } \ \partial B_{R} \cap \R^d_+ \ \hbox{ for all } \ R \geq R_1.\]
    Then by maximum principle in $B_{R} \setminus B_{1}$, for every $R \geq R_1$,
    \[v(y) \leq \ep + (\sup_{\partial B_1^+} |v|) |y|^{2-d+\delta} \ \hbox{ in } \ \R^d_+ \setminus B_{1}.\]
    Since $\ep>0$ was arbitrary
    \[v(y) \leq  (\sup_{\partial B_1^+} |v|) |y|^{2-d+\delta} \ \hbox{ in } \ \R^d_+ \setminus B_{1}.\]
      The lower bound argument is similar using the subsolution barriers from \lref{fundie-barriers}.
\end{proof}
\subsubsection{One-sided bounded solutions in $d=2$} The situation in $d=2$ is slightly different. First we show a maximum principle. In $d = 2$ bounded from below (resp. above) solutions of \eref{hodo-PDE} in the exterior of $B_1^+$ attain their minimum (resp. maximum) value on $\partial B_1^+$.
\begin{lemma}\label{l.limit-containment-2d}
        Suppose that $d =2$, $v$ solves \eref{hodo-PDE}, and $v$ is bounded from below then
    \[\inf_{\R^2_+ \setminus B_1} v(y) = \min_{\partial B_1^+} v,\]
    similarly, if $v$ is bounded from above then
    \[\sup_{\R^2_+ \setminus B_1} v(y) = \max_{\partial B_1^+} v.\]
\end{lemma}
\begin{remark}
    This is where we need the logarithmic barriers from \lref{logarithmic-barriers}. In particular we need a supersolution barrier $\phi(y)$ with $\phi(y) \to +\infty$ as $|y| \to +\infty$. Homogeneous supersolution barriers with a downward pointing singularity only exist with homogeneity $\alpha \leq d-2 = 0$, see \cite{ArmstrongSirakovSmart}. Thus we cannot achieve both the supersolution property and the property $\lim_{|y| \to \infty} \phi(y) = +\infty$ with a non-zero homogeneity as in \lref{fundie-barriers}, and we need to work with logarithmic barriers.
\end{remark}
\begin{proof}

We just consider the bounded from above case, the bounded from below case is similar.  By \lref{logarithmic-barriers} we have the radially symmetric barriers
\[\phi_L(|y|) := L^{-1}(\log |y| - \log(\log L + \log|y|) )\] 
which are supersolutions of \eref{hodo-PDE} in $\R^d_+ \setminus B_1$ as long as $L \geq R_0$. Here $R_0 >e$ is just some universal parameter determined in the proof of \lref{logarithmic-barriers}.

 Since $v$ is bounded from above $M := \sup_{\R^d_+ \setminus B_1} v < + \infty$. Suppose that 
 \[\alpha := \min\{M - \max_{\partial B_1^+} v,\tfrac{1}{2}\} >0,\]
 otherwise we are done. Let $r \gg 1$, and define $L(r) := \alpha^{-1}\log r$. Since $\log r \to +\infty$ as $r \to +\infty$, for $r>1$ sufficiently large $L(r) \geq R_0$ and also
 \begin{equation}\label{e.log-L(r)-stuff}
     L(r) \geq R_0 \ \hbox{ and also } \ \dfrac{\log(\log L(r) + \log r)}{\log r} \leq \dfrac{\alpha}{2} \leq \tfrac{1}{2}.
 \end{equation}

 For such large $r \gg 1$ define
\[\psi_r(y) = M-\frac{\alpha}{2}+ \phi_{L(r)}(|y|),\]
 which is a supersolution of \eref{hodo-PDE} in $\R^d_+ \setminus B_1$. On $y \in \partial B_r^+$, using \eref{log-L(r)-stuff},
\[\psi_r(y) = M-\frac{\alpha}{2}+\alpha(1-\dfrac{\log(\log L(r) + \log r)}{\log r}) \geq M \geq  v(y). \]
On $y \in \partial B_1^+$, using $r\gg 1$ from using \eref{log-L(r)-stuff} again,
\[\psi_r(y)  = M-\frac{\alpha}{2} -L(r)^{-1}\log\log L(r) \geq M-\alpha \geq \max_{\partial B_1} v \geq v(y). \]
So, by comparison, $v(y) \leq \psi_r(y)$ in $(B_r \setminus B_1)^+$. Evaluating at a fixed $y$, and using that $\phi_{L}(y) \to 0$ as $L \to \infty$ for fixed $y$,
\[v(y) \leq \lim_{r \to \infty} \psi_r(y) = M - \frac{\alpha}{2},\]
which contradicts the definition of $M$.
\end{proof}

Next we use the previous maximum principle with \lref{max-max-bd-2d} and Harnack inequality to establish quantitative growth bounds on exterior solutions in $d = 2$.

\begin{lemma}\label{l.one-sided-bound-2d}
    Suppose that $d =2$ and $v$ solves \eref{hodo-PDE}. There is $\eta_0(d)>0$ sufficiently small so that if $\sup_{\R^d_+ \setminus B_1} |\grad v| \leq \eta_0(d)$ and $v$ is bounded from below, then
    \[ 0 \leq v(y) -\min_{\partial B_1^+} v \leq   C(1+\log |y|) (\min_{\partial B_2^+} v - \min_{\partial B_1^+} v) \ \hbox{ for } \ |y| \geq 2. \]
     Similarly if $v$ is bounded from above then
     \[ 0 \geq v(y) -\max_{\partial B_1^+} v \geq  - C(1+\log |y|) (\max_{\partial B_1^+} v - \max_{\partial B_2^+} v) \ \hbox{ for } \ |y| \geq 2. \]
\end{lemma}

\begin{proof}
    We just do the bounded from below case. The inequality $v(x) \geq \min_{\partial B_1^+} v$ is the content of \lref{limit-containment-2d}.

    Now $w(y) = v(y) - \min_{\partial B_1^+} v$ is a non-negative solution of \eref{hodo-PDE}. We can apply \lref{max-max-bd-2d} to find, for all $ r \geq 2$,
    \[\min_{\partial B_r^+} w \leq C_\delta r^\delta (\min_{\partial B_2^+} v - \min_{\partial B_1^+} v).\]
    Since $w$ is non-negative Harnack inequality, \tref{hodograph-harnack}, implies that, for all $r \geq 2$,
    \[\max_{\partial B_r^+} w \leq C\min_{\partial B_r^+} w.\]
    Combining with the previous inequality completes the proof.
\end{proof}

\subsection{Quantitative linearization of the hodograph PDE}

Next we combine the growth (\lref{one-sided-bound-2d} in $d=2$) or decay (\lref{decay-est-d3} in $d \geq 3$) estimates with elliptic regularity to establish decay estimates on first and second order derivatives. Then, plugging these derivative estimates into the hodograph PDE, we obtain that $v$ (almost) solves the Laplace equation with Neumann boundary conditions up to an even more quickly decaying error.
\begin{lemma}\label{l.decay-barrier}
  Suppose $d \geq 2$. For any $\delta>0$ there is $\eta_0(\delta,d)>0$ small so that: if $v$ solves \eref{hodo-PDE} is bounded either from below or from above and $\sup_{\R^d_+ \setminus B_1} |\grad v| \leq \eta_0$ then for $k =1,2$
\begin{equation}\label{e.derivative-estimates-delta}
    |y|^{d-2+k-\delta}|\grad ^k v(y)| \leq C(\delta,k)\osc_{(B_2\setminus B_1)^+} v  \ \hbox{ in } \ \R^d \setminus B_{2}.
\end{equation}    In particular, by applying these derivative estimates in \eref{hodo-PDE}, it follows that 
    \[\begin{cases}
        \displaystyle|\Delta v(y)| \leq C(\delta)(\osc_{(B_2\setminus B_1)^+} v)^2|y|^{1-2d+\delta} & \hbox{in } \ \R^d_+ \setminus B_{2} \\
        \displaystyle 0 \leq \partial_{d}v(y) \leq C(\delta)(\osc_{(B_2\setminus B_1)^+} v)^2 |y|^{2-2d+\delta} & \hbox{on } \partial\R^d_+ \setminus B_2.
    \end{cases}\]
\end{lemma}
For the purposes of \tref{flat-exterior} it will suffice to use this result with $\delta = \frac{1}{2}$, but we keep $\delta>0$ as a parameter here because we believe it clarifies the important growth/decay rates in the present statement.
\begin{proof}
     First we apply the elliptic regularity estimates of \tref{hododgraph-regularity} to estimate the derivatives. For $y \in \R^d_+ \setminus B_2$ we have $v$ solving
\[
\begin{cases}
    \textup{tr}(A(\grad_y v)D^2_yv) = 0 & \hbox{ in } \ B_{|y|/2}(y)^+\\
    \partial_{y_d}v = \sqrt{1+|\grad_y 'v|^2}-1 & \hbox{ on } \partial \{y_d>0\} \cap B_{|y|/2}(y)^+
\end{cases}
\]
so \tref{hododgraph-regularity} yields 
\[|y|^k|\grad ^k v(y)| \leq C(\delta,k)\osc_{B_{|y|/2}(y)^+} v.\]
In $d \geq 3$ \lref{decay-est-d3} gives
\[\osc_{B_{|y|/2}(y)^+} v \leq C|y|^{2+\delta-d} \osc_{(B_2 \setminus B_1)^+} v,\]
while in $d=2$ \lref{one-sided-bound-2d} gives
\[\osc_{B_{|y|/2}(y)^+} v \leq C(1+\log|y|) \osc_{(B_2 \setminus B_1)^+} v.\]
Plugging these into the right hand side of the elliptic estimate gives the result.

    Next we plug in the elliptic estimates into the hodograph PDE \eref{hodo-PDE} and put the error terms on the right hand side for estimates on the Laplacian.  More specifically
    \[ \Delta v = \textup{tr}((I - A(\grad v))D^2v) \ \hbox{ in } \ \R^d_+ \setminus B_1.\]
    By \eref{hodograph-PDE-A-formula} and direct estimation
    \[|A(p) - \textup{I}| \leq C(d)|p| \ \hbox{ for } \ |p| \leq 1\]
    and so
    \[|\Delta v| \leq C|\grad v||D^2v| \ \hbox{ in } \ \R^d_+ \setminus B_1.\]
    So applying \eref{derivative-estimates-delta} with $k=1$ and $k=2$ and $\delta/2$ we find
    \[|\Delta v| \leq C(\osc_{(B_2 \setminus B_1)^+} v)^2|y|^{1+\delta/2 - d}|y|^{\delta/2 - d}\leq C (\osc_{(B_2 \setminus B_1)^+} v)^2|y|^{1 - 2d+\delta} \ \hbox{ in } \ \R^d_+ \setminus B_2.\]
    The argument for the Neumann condition is similar. We use \eqref{e.N-property} and \eref{derivative-estimates-delta} with $k=1$ and $\delta/2$ to find
    \[\partial_{y_d} v \leq C(\osc_{(B_2 \setminus B_1)^+} v)^2(|y|^{1+\delta/2 - d})^2 = C(\osc_{(B_2 \setminus B_1)^+} v)^2 |y|^{2+\delta - 2d}.\]

\end{proof}

\subsection{Precise asymptotics via Kelvin Transform}

We can now establish the asymptotic expansion at $\infty$ and, in particular, the existence of the capacity. The idea is to perform a Kelvin transform and then use estimate the difference with the harmonic replacement.  The proofs model the case of harmonic functions in exterior domains, using the quantitative linearization result \lref{decay-barrier} to show that the errors are sufficiently small.

\begin{proof}[Proof of \tref{flat-exterior}] For this proof denote $m:= \osc_{(B_2 \setminus B_1)^+} v < 1$, we are shortening the notation since this number will appear repeatedly in the estimates below.

    (Case $d \geq 3$.)  By \lref{height-exists} $s := \lim_{|y| \to \infty} v(y)$ exists, assume without loss that $s = 0$. Take the Kelvin transform of $v$
    \[\tilde{v}(z) = |z|^{2-d}v(\tfrac{z}{|z|^2}) \ \hbox{ on } \ z \in (B_1^+\cup B_1') \setminus \{0\}.\]
    Let us make a note that, since $v(y) \to 0$ as $|y| \to \infty$, then
    \begin{equation}\label{e.removable-sing-est-tildev}
        |z|^{d-2}\tilde{v}(z) \to 0 \ \hbox{ as } \ z \to 0.
    \end{equation}
    Which we will use to show that $\tilde{v}$ has a removable singularity at $z = 0$.

    We compute
    \[\grad \tilde{v}(z) = (2-d)\frac{z}{|z|^d}v(\tfrac{z}{|z|^2}) + \frac{1}{|z|^d}(I-2\frac{z \otimes z}{|z|^2})\grad v(\tfrac{z}{|z|^2}).\]
    Evaluating on $z_d = 0$, and using that $e_d\frac{z \otimes z}{|z|^2} = 0$ for such $z$,
    \[\partial_d\tilde{v}(z) = e_d \cdot \grad \tilde{v}(z) = \frac{1}{|z|^d}\partial_d v (\tfrac{z}{|z|^2}) \ \hbox{ on } \ z\in B_1' \setminus \{0\}.  \]
    Combining this with the estimates in \lref{decay-barrier} we find the following estimate on $z \in B_{1/2}' \setminus \{0\}$
  \begin{equation}\label{e.kelvin-neumann-est}
      |\partial_{d}\tilde{v}(z)| = |z|^{-d}|\partial_d v (\tfrac{z}{|z|^2})| \leq Cm^2|z|^{-d}|z|^{2d-2-\delta} = Cm^2|z|^{d-2-\delta}.
  \end{equation}
    Also recall the following formula for the Laplacian under the Kelvin transform
    \[\Delta \tilde{v}(z) = |z|^{-d-2}\Delta v(\tfrac{z}{|z|^2}).\]
    Combining this formula with \lref{decay-barrier} we have the following PDE in $B_{1/2}^+$ 
    \begin{equation}\label{e.kelvin-laplace-est}
        |\Delta \tilde{v}(z)| = |z|^{-d-2}|\Delta v(\tfrac{z}{|z|^2})| \leq Cm^2|z|^{-d-2}||z|^{-2}z|^{1-2d+\delta} = Cm^2|z|^{d-3-\delta}.
    \end{equation}
    We summarize the two previous estimates, \eref{kelvin-neumann-est} and \eref{kelvin-laplace-est}, in the following PDE 
    \begin{equation}\label{e.tilde-v-pde-est}
        \begin{cases}
            \Delta \tilde{v}(z) = \tilde{f}(z)  & \hbox{in } B_{1/2}^+\\
            \partial_{d}\tilde{v}(z) = \tilde{g}(z) &\hbox{on } B_{1/2}' \setminus \{0\}
        \end{cases}
    \end{equation}
    where
    \begin{equation}\label{e.tilde-v-pde-est-rhs}
        |\tilde{f}(z)|  \leq Cm^2|z|^{d-3-\delta}  \hbox{ on } B_{1/2}^+ \hbox{ and }  |\tilde{g}(z)| \leq Cm^2|z|^{d-2-\delta} \hbox{ on } B_{1/2}' \setminus \{0\}.
    \end{equation}

    Now we need to argue that $\tilde{v}$ has a removable singularity at $z=0$.  We subtract off the Neumann kernel of the right hand side of \eref{tilde-v-pde-est}
    \[\psi(z) = \int_{B_{1/2}^+} \Phi(z-w) \tilde{f}(w) \ dw + \int_{B_{1/2}'} \Phi(z-w)\tilde{g}(w) dS(w)\]
where 
\[\Phi(z) = \frac{1}{d(d-2)|B_1|}(|z|^{2-d}+|(z',-z_d)|^{2-d})\]
is the standard Neumann kernel for $\R^d_+$. The bounds \eref{tilde-v-pde-est-rhs} imply that 
\begin{equation}\label{e.psi-at-zero}
    |\psi(z)-\psi(0)| \leq Cm^2|z| \ \hbox{ in } \ B_{1/2}^+.
\end{equation}
    Next let $\bar{v}$ be the solution of
    \[\begin{cases}
        \Delta \bar{v} = 0 & \hbox{in } B_{1/2}^+\\
        \partial_d \bar{v} = 0 &\hbox{on } \partial \R^d_+ \cap B_{1/2}\\
        \bar{v}(z) = \tilde{v}(z) - \psi(z) &\hbox{on } (\partial B_{1/2})^+.
    \end{cases}\]
    We claim that
    \begin{equation}
        \tilde{v}(z) \equiv \bar{v}(z) + \psi(z) \ \hbox{ in } \ B_{1/2}^+.
    \end{equation}
The reasoning is that (with even reflection)
\[w(z) = \tilde{v}(z)-(\bar{v}(z) + \psi(z))\]
is harmonic in $B_{1/2} \setminus \{0\}$, zero on the boundary $\partial B_{1/2}$, and, by \eref{removable-sing-est-tildev} and \eref{psi-at-zero}, grows more slowly than $|z|^{2-d}$ at the origin so it is zero. 

By even reflection and interior Lipschitz estimates of harmonic functions for $z \in B_{1/4}^+$
\[|\bar{v}(z) - \bar{v}(0)| \leq C|z|\osc_{(\partial B_{1/2})^+}(\tilde{v} - \psi) \leq  Cm|z|\]
using that, by \eref{psi-at-zero} and $m \leq 1$, $\osc_{(\partial B_{1/2})^+}\psi \leq Cm^2 \leq Cm$, and that 
\[\osc_{(\partial B_{1/2})^+}\tilde{v} =2^{2-d}\osc_{(\partial B_{2})^+}v \leq 2^{2-d}m.\]
    
    So now we conclude that
    \[|\tilde{v}(z) -\tilde{v}(0)| \leq |\psi(z) - \psi(0)| + |\bar{v}(z) - \bar{v}(0)| \leq  Cm|z| \ \hbox{ in } \ B_{1/4}^+.\]
    Calling $k:= \tilde{v}(0)$ we also find
    \[|k| = |\tilde{v}(0)| = |\bar{v}(0) + \psi(0)| = |\dashint_{\partial B_{1/2}^+} (\tilde v - \psi)(z)dS(z) +\psi(0)| \leq Cm.\]
     Undoing the Kelvin transform we arrive at
    \begin{align*}
        |v(y) - k|y|^{2-d}| &= ||y|^{2-d}\tilde{v}(|y|^{-2}y)-\tilde{v}(0)|y|^{2-d}|\\
        &=|y|^{2-d}|\tilde{v}(|y|^{-2}y)-\tilde{v}(0)|\\
        &\leq|y|^{2-d}Cm|y|^{-1} = Cm|y|^{1-d}.
    \end{align*}

    (Case $d=2$)   In the case $d=2$ we argue similarly, with inversion, and use a typical complex analysis trick. Again this is following a classical argument for harmonic functions in exterior domains, with an additional error term controlled via \lref{decay-barrier}. We will now use $z = x+iy$ for the complex variable.
    
    Define
    \[\tilde{v}(z) := v(\tfrac{1}{z}) \ \hbox{ for } \ z \in B_1\setminus \{0\}.\]
    Define, as before,
        \[\psi(z) = \int_{B_{1/2}^+} \Phi(z-w) \tilde{f}(w) \ dw + \int_{B_{1/2}'} \Phi(z-w)\tilde{g}(w) dS(w)\]
where 
\[\Phi(z) = \frac{-1}{4\pi}(\log |z|+\log |(z',-z_d)|)\]
    and $\bar{v}$ solving
    \[\begin{cases}
        \Delta \bar{v} = 0 & \hbox{in } B_{1/2}^+\\
        \partial_d \bar{v} = 0 &\hbox{on } \partial \R^d_+ \cap B_{1/2}\\
        \bar{v}(z) = \tilde{v}(z) - \psi(z) &\hbox{on } (\partial B_{1/2})^+.
    \end{cases}\]
    Then call
    \[w(z) = \tilde{v}(z)-(\bar{v}(z) + \psi(z)).\]
    The even reflection of $w$, not relabeled, is harmonic in the punctured disk $B_1 \setminus \{0\}$. Thus we can write, for some $k \in \R$,
    \[w(z) = k \log 2|z| + \textup{Re}(h(z))\] 
    where $h$ is holomorphic in the punctured disk. Since $w$ is $o(|z|^{-1})$ at the origin and so is $\log |z|$, then so is $|\textup{Re}(h(z))|$. Thus $0$ is a removable singularity and $h$ is holomorphic in the entire $B_1$. Since $\textup{Re}(h(z)) = w(z) -k\log (2\cdot\frac{1}{2})= 0$ on $\partial B_{1/2}$ then $h \equiv 0$.  Thus
    \[\tilde{v}(z) = \bar{v}(z) + \psi(z)+k\log 2 +k \log|z|\]
and so
    \[\tilde{v}(z) = k\log |z| + \bar{v}(z) + \psi(z)\]
    and now we conclude since $\bar{v}$ and $\psi$ are continuous at $z = 0$.
\end{proof}

\section{General solutions with a single-site defect}\label{s.general-asymptotics}
In this section we consider solutions of the single-site defect problem in the entire space, namely
\begin{equation}\label{e.defect-prob}
     \begin{cases}
         \Delta u = 0 & \hbox{in } \{u>0\} \\
         |\grad u| = 1+q(x) &\hbox{on } \partial \{u>0\}, 
     \end{cases}
\end{equation}
where the defect $q(x): \R^d \to (-1,\infty)$ is smooth and supported in $\overline{B_1(0)}$. For applications in the rest of the paper, we focus primarily on ``proper" solutions.

Our main result in this section regards the far-field asymptotic expansion  of {proper} solutions to \eref{defect-prob}. In later sections we will need to establish that the single-site solutions of interest are indeed {proper}.

\begin{theorem}\label{t.asymptotic-expansion-at-infty}
   
   Let $u$ be a {proper} solution of \eref{defect-prob}. 
   \begin{enumerate}[label = (\roman*)]
    \item In the case $d=2$. There is $ k(u)\in \R$ so that:
    \[\sup_{x \in \overline{\{u>0\}}}|u(x) -(x_d  +k \log |x|)| < +\infty \] 
       and there are universal constants $\sigma_0 \in (0, 1/2]$ and $C\geq 1$ so that  
  \[  \hbox{ if } |q|\leq \sigma_0 \hbox{ then } \  C\min q \leq k(u) \leq C \max q.\]
       Furthermore, for general $q \in C_c(B_1; (-1,\infty))$, there is $C(\min q)$ universal so that $k(u) \geq -C$.
       \item In the case $d \geq 3$. There are $s(u), k(u) \in \R$ so that:
   \[|u(x) -(x_d +  s -k|x|^{2-d})| \leq C_0(1+|x|)^{1-d} \ \hbox{ for } \ x \in \{u>0\}\]
   and if $k \neq 0$ then  
   \[|s|,|k| \leq C_0=C_0(d,\min q,\max q).\]
   Furthermore, if $\max |q| \leq \sigma_0(d)$, then $C_0\leq C(d) \max |q|$.
   \end{enumerate}
\end{theorem}

Note that we lack universal control on the higher order error terms in the asymptotic expansion in $d=2$, it remains open whether it can be achieved.

We present the proof of this theorem below in \sref{asymptotic-expansion-at-infty}. First, since $u$ is {proper}, at some large (non-quantitative) scale it is sufficiently flat to employ \tref{flat-exterior-original-coord}. Still, we need quantified information on $k$, $s$ (and $R$ for $d\geq 3)$ for applications for the rest of the paper. This is a significant new challenge,  for which we need barriers with the correct asymptotic expansion. Although the barrier constructions are quite concrete, they play an essential role in the theory later.

Let $\sigma \in (-1,+\infty)$. Call $B_1' = \{x' \in \R^{d-1}: |x'| <1\}$ and $B_1' \times \R$ is the cylinder above $B_1'$ with axis in the $e_d$ direction. The goal is to construct subsolutions, in case $\sigma>0$, and supersolutions, in case $\sigma<0$, of the problem
\begin{equation}\label{e.barrier-eqn}
    \begin{cases}
         \Delta \phi = 0 & \hbox{in } \{\phi>0\} \\
         |\grad \phi| = 1+ \sigma{\bf 1}_{B_1' \times \R}&\hbox{on } \partial \{\phi>0\},
     \end{cases}
\end{equation}
with the asymptotic bounds, for some $C_1(\sigma,d)>0$, in $d \geq  3$
\begin{equation}\label{e.barrier-expansion-3d}
\begin{array}{c}
    0 \leq \textup{sgn}(\sigma)(x_d-\phi(x)) \leq  C_1(\sigma,d) \min\{1,|x|^{2-d}\} \ \hbox{ in } \ \{\phi >0\}.
    \end{array}
\end{equation}
and in $d=2$,
\begin{equation}\label{e.barrier-expansion-2d}
\begin{array}{c}
     \textup{sgn}(\sigma)(x_d-\phi(x)) \leq  -C_1(\sigma) \max\{1,\log |x|\} \ \hbox{ in } \ \{\phi >0\}.
    \end{array}
\end{equation}
 Note that \eref{barrier-eqn} is invariant with respect translations in the $x_d$ variable, this allows the barriers to be used in sliding comparison arguments.

We will construct three different types of barriers in this section. In the nearly linearized regime when $\max q$ and/or $\min q$ are small we can construct barriers via patching a linear function, in $B_1$, with a slightly tilted fundamental solution outside of $B_1$. This is a little bit more difficult in dimension $d=2$ where we were unable to successfully construct the barrier directly in original coordinates, and instead use hodograph coordinates again. In this almost linear regime the sub and supersolution constructions are basically symmetrical. See \sref{small-sigma-3d} and \sref{small-sigma-2d} for these constructions. In the nonlinear regime, when $\min q$ and/or $\max q$ are large, we have two quite distinct barrier constructions for sub and supersolutions.  This asymmetry reflects an important asymmetry between advancing and receding regimes in the truly nonlinear problem. The supersolution constructions works in all $d \geq 2$, but the subsolution construction only works in $d \geq 3$.  We are very interested whether analogous subsolutions exist in the nonlinear regime in dimension $d=2$. See \sref{barrier-subsolution} and \sref{barrier-supersolution} for the sub and supersolution constructions respectively.

We state the results of the barrier constructions, first in the nearly linear regime and then in the nonlinear regime.

\begin{proposition}\label{p.barrier-prop} 

\begin{enumerate}[label = (\roman*)]
    \item There is a universal constant $0<\sigma_0<1$ such that, for $|\sigma| \leq \sigma_0$ and $\sigma>0$ (resp. $\sigma <0$) there is a smooth subsolution (resp. supersolution) $\phi_\sigma$ of \eref{barrier-eqn} satisfying \eref{barrier-expansion-3d} (or \eref{barrier-expansion-2d}) with $C_1(d,\sigma) = C(d)|\sigma|$.
    \item In all $d \geq 2$ and for any $-\infty < \sigma < 0$ there is a smooth supersolution of \eref{barrier-eqn} $\phi_\sigma$ satisfying \eref{barrier-expansion-3d} (or \eref{barrier-expansion-2d}).
    \item In $d \geq 3$ and for any $0 < \sigma <\infty$ there is a smooth subsolution of \eref{barrier-eqn} $\phi_\sigma$ satisfying \eref{barrier-expansion-3d}.
\end{enumerate}
\end{proposition}

\subsection{Asymptotics of general solutions}\label{s.asymptotic-expansion-at-infty} Before proceeding to the barrier constructions in \sref{barrier-constructions} we show how to derive \tref{asymptotic-expansion-at-infty} using the barriers from \pref{barrier-prop}.

\begin{proof}[Proof of \tref{asymptotic-expansion-at-infty}]
Note that if $\partial \{u>0\} \cap \overline{B_1(0)} = \emptyset$ then $u$ globally solves the homogeneous Bernoulli problem \eref{bernoulli} and is proper and therefore $u(x) = (x_d+s)_+$ for some $s$. Thus for the remainder of the proof we can assume that 
\begin{equation}\label{e.bf-touches-b1}
    \partial \{u>0\} \cap \overline{B_1(0)} \neq \emptyset.
\end{equation}

{\bf Step 1.} By hypothesis, for all $\ep>0$ there is $R(\ep)$ sufficiently large so that for all $r \geq R(\ep)$
  \[(x_d-\ep)_+ \leq r^{-1}u(rx) \leq (x_d+\ep)_+ \ \hbox{ in } \ B_1(0).\]
  Let $\eta_1(d)>0$ be from \tref{flat-exterior-original-coord}. Then \lref{blow-down-grad} implies that there is $R_1>1$ sufficiently large so that $u\in C^2(\{u>0\}\setminus B_{R_1}(0))$ and 
\[|\grad u(x) - e_d| \leq \eta_1 \ \hbox{ in } \ \{u>0\} \setminus B_{R_1}(0).\]
 \tref{flat-exterior-original-coord} in turn implies that, in $d \geq 3$,
 \begin{equation}\label{e.qual-u-limit}
    u(x) = x_d + s - k|x|^{2-d} + E(x) \ \hbox{ in } \ \{u>0\} \setminus B_{R_1} \ \hbox{ with } \ |E(x)| \leq C(1+|x|)^{1-d}
\end{equation}
while in $d=2$
 \begin{equation}\label{e.qual-u-limit-2d}
    u(x) = x_d +k\log |x| + E(x) \ \hbox{ in } \  \{u>0\} \setminus B_{R_1} \ \hbox{ with } \ \sup|E(x)| < +\infty .
\end{equation}
Note that so far $s$, $k$, and the error term $E(x)$ depend on $R_1$, and $R_1$ depends on the solution $u$ in a non-universal way.

 {\bf Step 2.}  Next we use  $s$ in a sliding barrier argument to control $R_1$ quantitatively. We present separate arguments for $d\geq 3$ and $d=2$.

{ (Case: $d \geq 3$).} Let $\phi=\phi_{\sigma}$ be the outer regular supersolution barrier from \pref{barrier-prop} with $\sigma=\min q \in (-1,0)$. By \eqref{e.barrier-expansion-3d} we have
\[(x_d)_+ \leq \phi(x) \leq (x_d + C_1\min\{1,|x|^{2-d}\} )_+.\]
Now we apply a sliding argument with the family $\phi(x+te_d)$. Note that $\phi(x+te_d) \geq (x_d+t)_+ > u(x)$ in $\overline{\{u>0\}}$ for sufficiently large $t>0$. Also for every $t>s$, by \eref{qual-u-limit} we have $\phi(x+te_d) \geq (x_d+t)_+ > u(x)$ in $\overline{\{u>0\}}$ for $x \in \R^d \setminus B_{R(t)}$.  So by \lref{sliding-comparison} $u(x) \leq \phi(x+te_d)$ for all $t \geq s$ and so we conclude
\begin{equation}\label{e.u-barrier-UB-3}
u(x) \leq (x_d+s+C_1\min\{1,|x+se_d|^{2-d}\})_+.
\end{equation}
We can also bound $s$, using \eref{bf-touches-b1}. So letting $x^0 \in \partial \{u>0\} \cap \overline{B_1}$ be a point realizing \eref{bf-touches-b1}, then
\[x^0_d+s+C_1\min\{1,|x^0+se_d|^{2-d}\} \geq 0\]
and therefore
\[s \geq -C_1-1.\]

Similar arguments using the subsolutions $\phi_{\sigma}$ with $\sigma = \max q>0$ from \pref{barrier-prop} yields
\begin{equation}\label{e.u-barrier-LB-3}
    u(x) \geq (x_d+s-C_2\min\{1,|x+se_d|^{2-d}\})_+ \ \hbox{ and } \ s \leq C_2+1
\end{equation}
with a $C_2$ depending only on $d$ and on $\max q$.  The inequalities \eref{u-barrier-LB-3} and \eref{u-barrier-UB-3} also imply that
\[-C_2 \leq k \leq C_1.\]

Thus, by using the quantitative flatness \eref{u-barrier-UB-3} and\eref{u-barrier-LB-3} in \tref{flat-implies-c1alpha}, there is a radius $R_0=R_0(C_1,C_2, \eta_1)=R_0(d, \min q, \max q)$ so that 
\begin{equation}\label{Lip_bound} |\grad u - e_d| \leq \eta_1 \ \hbox{ on } \ \{u>0\} \setminus B_{R_0}(0).
\end{equation}
As before, but now with quantified $R_0$,  \tref{flat-exterior-original-coord} implies that, for $x \in \{u>0\} \setminus B_{R_0}$,
\begin{align*}
    |x|^{d-1}|u(x) - (x_d+s+k|x|^{2-d})| &\leq C\osc_{(B_{2R_0} \setminus B_{R_0}) \cap \{u>0\}} (u(x) - x_d) \\
    &
    \leq C(d,\min q,\max q)
\end{align*}
where the last inequality again uses \eqref{Lip_bound}.

 { (Case: $d =2$)}. Here we assume that $\max |q| \leq \sigma_0$ to use the barriers given in \pref{barrier-prop}. Let $\phi_{\sigma}$ be as given in  \pref{barrier-prop} with $\sigma = \min q \in (-1,0] $ satisfying \eqref{e.barrier-eqn} and \eqref{e.barrier-expansion-3d}. We claim that $k \geq -C_1$. Suppose otherwise, $k < - C_1$. Since $u$ is {proper} it is one-sided flat, since $k < -C_1$ in \eref{qual-u-limit} it must be flat from above, $u(x) \leq (x_d+T)_+$ for some $T \in \R$. We perform a sliding comparison with $\phi(x+te_d)$. For $t > T$ then $\phi(x+te_d) > u(x)$ in $\overline{\{u>0\}}$.  Let $k < k' < -C_1$. For any $t \in \R$ there is $R$ sufficiently large so that, for $x \in \{u>0\} \setminus B_R$,
\[\phi(x+te_d) \geq (x_d+t-C_1\max\{1,\log |x+te_d|\})_+ > (x_d+k'\log |x|)_+ > u(x). \]  Then Lemma ~\ref{l.sliding-comparison} implies $\phi(x+te_d) \geq u(x)$ for all $t$, yielding a contradiction.

Arguing similarly with the subsolution barriers from \pref{barrier-prop} shows that $k \leq C_1$.

\end{proof}

\subsection{Barriers}\label{s.barrier-constructions}
 In this section we construct the barriers presented in Proposition~\ref{p.barrier-prop}.
The construction of the barriers are given in the order of increasing difficulty. We begin with the simplest construction, which is for small $\sigma$, i.e. small $\|q\|_{L^\infty}$, in $d=3$.

For small $\sigma$, our barrier construction is relatively simple, by patching of inner and outer parts that are  $O(\sigma)$-perturbations of the planar profile, based on the Laplace fundamental solution and its derivative. The barrier construction in $d=2$ is in hodograph coordinates, since we can only construct the two dimensional fundamental solution type barrier in that coordinate system. This approach only works when $\sigma$ is sufficiently small, below a universal threshold.  This is, perhaps, natural since the construction views subsolutions and supersolutions more or less symmetrically.

Both the nonlinearity and the asymmetry between advancing and receding become more severe for large $\sigma$. Our perturbations are more nonlinear in this case, and involve planting a sizable sink and source term, respectively, to pull or push the planar profile.  Furthermore the construction of subsolution and supersolution barriers are no longer symmetrical and have a slightly different geometry.

\subsubsection{Barriers for small defects: $d\geq3$}\label{s.small-sigma-3d}

Let $C>1$ to be chosen sufficiently large depending on dimension and $|\sigma| \leq \sigma_0$, $\sigma>0$ in the subsolution case and $\sigma<0$ in the supersolution case. We define the inner solution
\[\phi_{in}(x) := (1+\sigma)x_d-C\sigma\]
and the outer solution
\[\phi_{out}(x) := x_d -C\sigma|x|^{2-d}+ \sigma \frac{x_d}{|x|^d}.\]
Note that $x_d/|x|^d$ is a constant multiple of $\partial_{x_d} \Phi(x)$, and thus is harmonic away from the origin. Therefore $\phi_{out}$ is harmonic away from the origin and $\phi_{in}$, being linear, is harmonic everywhere.  Define the patching
\[\phi(x):= \begin{cases}
    \phi_{in} (x) & |x| < 1\\
    \phi_{out} (x) & |x| \geq 1.
\end{cases}\]
Note that on $|x| = 1$
\[\phi_{out}(x) = x_d-C\sigma+\sigma x_d = \phi_{in}(x).\]
Thus $\phi$ is continuous across the patch on $\partial B_1$. The proposed subsolution / supersolution of the Bernoulli problem will be $\phi(x)_+$.

\begin{lemma}
    For $|\sigma| \leq \sigma_0(d)$ the function $\phi(x)_+$ is a subsolution of \eref{barrier-eqn}, in the case $\sigma>0$, and is a supersolution of \eref{barrier-eqn}, in the case $\sigma<0$.
\end{lemma}

\begin{proof}
We claim that, in the subsolution case $\sigma>0$, $\phi(x) = \max\{\phi_{in}(x),\phi_{out}(x)\}$ in a neighborhood of $\partial B_1$. It suffices to show that
\[x\cdot \nabla \phi_{in }(x) < x\cdot \nabla \phi_{out}(x) \ \hbox{ on } \ x \in \partial B_1.\]
We will check this by direct computation. First we record
\begin{equation}\label{e.grad-phiout}
    \grad \phi_{out} = e_d+\sigma\bigg[C(d-2) \frac{x}{|x|^d}+ \frac{e_d}{|x|^d} -d\frac{x_dx}{|x|^{d+2}} \bigg].
\end{equation}
So on $x \in \partial B_1$
\[x\cdot \nabla \phi_{out}(x) = x_d+\sigma[C(d-2)+(1-d)x_d] \ \hbox{ and } \ x\cdot \nabla \phi_{in }(x) = (1+\sigma)x_d.\]
So fixing $C = \frac{d}{d-2}+1$ then, on $x \in \partial B_1$ using that $1 \geq x_d$ on that set,
\[x\cdot \nabla \phi_{out}(x) > x_d + \sigma[ d+(1-d)x_d] \geq x_d + \sigma[ d x_d+(1-d)x_d] = (1+\sigma)x_d = x\cdot \nabla \phi_{in }(x).\]
By a symmetrical argument, in the supersolution case $\sigma<0$, $\phi(x) = \min\{\phi_{in}(x),\phi_{out}(x)\}$ in a neighborhood of $\partial B_1$.

Next we check the free boundary condition. The free boundary condition inside $B_1$ is immediate since the solution is linear, so we only need to check for the outer solution. Then, using \eref{grad-phiout} and expanding the quadratic, 
\[|\grad \phi_{out}|^2 = 1 + 2\sigma\bigg[\frac{1}{|x|^d}+C(d-2) \frac{x_d}{|x|^d} -d\frac{x_d^2}{|x|^{d+2}} \bigg]+O\left(\frac{\sigma^2}{|x|^{2(d-1)}}\right)\]
Let us use that 
\[x_d = C \sigma |x|^{2-d} \frac{1}{1-\sigma|x|^{-d}} \ \hbox{ on } \ \partial \{\phi >0\} \setminus B_1.\]
So then, for $\sigma \leq 1$ and $|x| \geq 1$,
\[|\grad \phi_{out}|^2 = 1 + 2\sigma\frac{1}{|x|^d}+O\left(\frac{\sigma^2}{|x|^{2(d-1)}}\right).\]
Note that $2(d-1) > d$ in dimensions $d\geq 3$.  Then $|\grad \phi_{out}|^2 >1$ for $|x| \geq 1$ as long as we choose $0 < \sigma \leq \sigma_0(d)$ with sufficiently small $\sigma_0$ depending on dimension.  The supersolution case $-\sigma_0 \leq \sigma < 0$ is symmetrical.

Finally note that
\[ (x_d - (C+1)\sigma)_+ \leq \phi(x)_+ \leq (x_d + (C+1)\sigma)_+.\]
\end{proof}

\subsubsection{Barriers for small defects: $d=2$}\label{s.small-sigma-2d}
 Here we will utilize the exterior barrier from \lref{logarithmic-barriers} in hodograph coordinates, to construct  barriers. While our construction is similar to  the patched barrier in higher dimensions from \sref{small-sigma-3d}, we face additional technical challenges here since we need to smooth out the solution at the patching in order to invert the hodograph transform. We go around this with standard mollifier and keep the computations to minimum.

 Let us first point out the issue in original coordinates via a Lemma, which we will use later for other purposes.  
 
 \begin{lemma}\label{l.logarithmic-supersolutions}
The function
 \[\phi(x) :=(x_d+\sigma \log |x|+ s)_+\]
 is a supersolution of \eref{bernoulli} in $\R^d \setminus B_3$ as long as $\sigma s \geq 0$.
 \end{lemma}
 \begin{proof}
     Note that $\phi$ is harmonic in its positivity set away from $0$. The zero set is on
\[0 = \phi (x) = x_d + \sigma \log |x|+s\]
and the slope is
\[|\grad \phi(x)|^2 = 1  + 2 \sigma\frac{x_d}{|x|^2} + \sigma^2\frac{1}{|x|^2}.\]
Evaluating the slope on the zero set, by plugging in $x_d = -\sigma\log |x|$,
\[|\grad \phi(x)|^2 =1-2\frac{1}{|x|^2}\left[\sigma^2\log |x|+\sigma s - 1\right].\]
So, since $\log 3>1$ and $\sigma s \geq 0$, we conclude that $|\grad \phi(x)|^2 < 1$ on $\partial \{\phi>0\} \setminus B_3$.
 \end{proof}
 This is actually fine for the purposes of \tref{asymptotic-expansion-at-infty} in the case $\sigma<0$, but, unfortunately, this is the wrong direction in the case $\sigma>0$.  We have not discovered any elementary way to fix this in the original coordinates.

We will work in the same hodograph coordinate system which was introduced in Section \ref{s.hodograph}. We will denote $y$ the variable in hodograph coordinates and $x$ the variable in standard coordinates as we did before.  Note that our conventions for the hodograph transform does switch the role of sub and supersolutions. 

Let $|\sigma| \leq \sigma_0$ a sufficiently small constant to be specified via the computation. We will define an inner solution, to be used inside $B_2(0)^+$, and an outer solution, to be used outside $B_2(0)^+$. In the case $\sigma>0$ we will construct a supersolution, and in the case $\sigma<0$ we will construct a subsolution of the hodograph equation
\begin{equation}\label{e.hodo-PDE-barr}
    \begin{cases}
          \textup{tr}(A(\grad_y v)D^2_yv) =0 & \hbox{ in } \ \{y_d>0\}, \\
    (1+\sigma {\bf 1}_{B_{1/2}})(1+\partial_{y_d}v) =\sqrt{1+|\grad_y 'v|^2}  & \hbox{ on } \partial \{y_d>0\}. 
    \end{cases}
\end{equation}
This is analogous to \eref{barrier-eqn} on the hodograph side. The choice of putting the defect in $B_{1/2}$ is for convenience, so we can patch on $\partial B_1$ with a little room. We will need to construct smooth, at least $C^1$, sub/supersolutions in order to invert the hodograph transform. But we begin with a non-smooth construction.

First define, using $|\sigma| < 1$, 
\[\varsigma(\sigma) := -\frac{\sigma}{1+\sigma} \ \hbox{ so that } \ (1+\sigma)(1+\varsigma) = 1.\]
Define the inner solution
\[\psi_{in}(y):=\varsigma y_d  \]
and the outer solution
\[\psi_{out}(y):= 2\varsigma (\log|y| + \log(1+\log|y|)) +\varsigma \frac{ y_d}{|y|^2}. \]
Then define the patched solution
\begin{equation}\label{patched}
    \psi(x):= \begin{cases}
        \psi_{in}(y) & y \in B_1(0)\\
        \psi_{out}(y) & y \not\in B_1(0).
    \end{cases}
\end{equation}
Notice that $\psi$, thus defined, is continuous on $\R^d$. 

As in the previous section, we will need to check interior and boundary sub and supersolution conditions (see \lref{logarithmic-barriers}) as well as the correct gradient discontinuity at the patching on $\partial  B_1$. Finally we will need to do an inverse hodograph transform to get a sub/supersolution in the original coordinates. Implementing this idea is technically tricky, since the hodograph and inverse hodograph transform require $C^1$ regularity. We achieve this regularity by using a standard  mollifier.  {We skip the proof since it is standard but lengthy.\footnote{For more details see \url{https://doi.org/10.48550/arXiv.2512.11152}}}

\begin{lemma}\label{l.2d-smoothed-barriers}
    Let $|\sigma | \leq \sigma_0 < 1$, $\varsigma: = -\sigma/(1+\sigma)$, and $0 < \ep \leq \frac{1}{2}$, and let $\eta_\e(x):= \e^{-d} \eta(\e x)$, where $\eta$ is the standard radially symmetric mollifier supported in $B_1(0)$. If $\sigma_0>0$ is sufficiently small, then for $\varsigma>0$ (resp. $\varsigma <0$) $\psi(x)$ given in \eqref{patched} and $\eta_\e *\psi$ are  subsolutions (resp. supersolutions) of \eref{hodo-PDE-barr}. 
\end{lemma}

Then, to conclude the proof of the relevant piece of \pref{barrier-prop} we can just apply an inverse hodograph transform.
\begin{corollary}\label{c.2d-smoothed-inverse-hodo}
    Let $\tilde{\psi}_\sigma := \eta_{1/2} * \psi_\sigma$ and let $\phi_\sigma$ be the inverse hodograph transform of $\tilde{\psi}_\sigma$. Then, up to a hyperbolic rescaling, $\phi_\sigma$ is a supersolution (case $\sigma<0$) or subsolution (case $\sigma >0$) of \eref{barrier-eqn} satisfying the bound \eref{barrier-expansion-2d} with $C_1(d,\sigma) = C(d) |\sigma|$.
\end{corollary}

\subsubsection{Barriers for large defects: subsolutions in $d \geq 3$}\label{s.barrier-subsolution} To achieve a subsolution, we will perturb the half-plane solution $(x_d)_+$ by putting a large sink term along a line segment 
\[L_R:= \{x \in \R^d: |x'|=0, -R\leq x_d \leq R\}.\] 
Namely we consider, for $d\geq 3$,
\begin{equation}\label{e.line-subsolution}
\phi_R(x):= \left(x_d - \int_{L_R} |x-y|^{2-d} d\mathcal{H}^1(y)\right)_+.
\end{equation}
It is straightforward to check that
\[\phi_R(x) = x_d - 2R|x|^{2-d}+O(|x|^{1-d}) \ \hbox{ as } \ \{\phi_R>0\} \ni x \to \infty.\]

\begin{lemma}\label{lem:subsolution}
For any $\sigma>0$ and $d\geq 3$, there are $R$, $a$, and $z$ depending on $\sigma$ and $d$ such that $\phi(x):= a^{-1}\phi_R(ax+ze_d)$ is a supersolution of \eref{barrier-eqn}. 
\end{lemma} 

Note that, although we can define a similar barrier in dimension $d=2$ it does not seem to satisfy the subsolution condition at the free boundary. The following proof relies on the fact that the free boundary of $\phi_R$ is contained in $x_d \geq 0$, this cannot be the case any longer in $d=2$ since for any solution with nontrivial capacity the free boundary will need to grow logarithmically into $x_d < 0$.
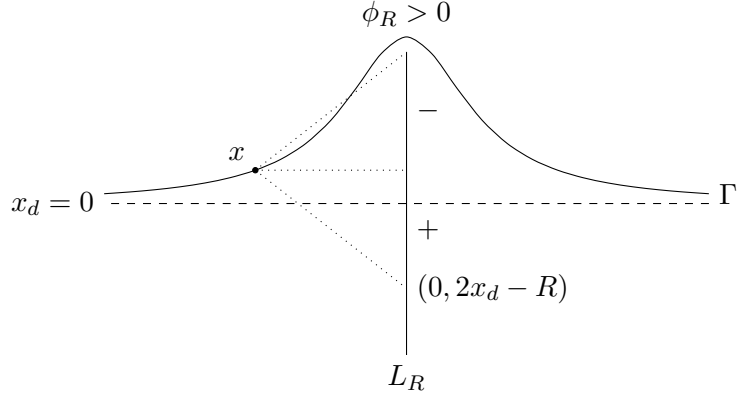
\begin{figure}
    \centering
    \begin{tikzpicture}[scale = 2]
    \draw (0,1) -- (0,-1) node[below]{$L_R$};
    \draw[dashed] (2,0) -- (-2,0) node[left] {$x_d = 0$};
    \draw[domain=-2:2,smooth,variable=\x]
  plot ({\x},{1.1/(1+4*\x*\x)}) node[right] {$\Gamma$};
  \def\xx{-1};
  \def\yy{{1.1/(1+4*\xx*\xx)}};
  \def\zz{{2*1.1/(1+4*\xx*\xx)-1}};
  \def\minusloc{{1.1/(1+4*\xx*\xx)+.5*(1-1.1/(1+4*\xx*\xx))}};
  \def\plusloc{{1.1/(1+4*\xx*\xx)-.5*(1-1.1/(1+4*\xx*\xx))}};
  \filldraw (\xx,\yy) circle (.5pt) node[above left] {$x$};
  \draw[dotted] (\xx,\yy) -- (0,1);
  \draw[dotted] (\xx,\yy) -- (0,\zz) node[right]{$(0,2x_d-R)$};
  \draw[dotted] (\xx,\yy) -- (0,\yy);
  \node[above] at (0,1.1) {$\phi_R>0$};
  \node[right] at (0,\minusloc) {$-$};
  \node[right] at (0,\plusloc) {$+$};
\end{tikzpicture}
    \caption{Diagram showing the free boundary $\Gamma$ of the barrier $\phi_R$, the line $L_R$ where the sources are placed, and the geometry of the integral cancellation which leads to the subsolution property.}
    \label{f.subsolution-barrier}
\end{figure}
\begin{proof} 
By construction the positive set does not contain $L_R$, and so $\phi_R$ is harmonic in its positive set.  Now we need to measure $D\phi_R$ on the free boundary $\Gamma:=\partial\{\phi_R>0\}$ to conclude that this is a subsolution. Also note that, since $\phi_R(x) < x_d$, the closure of the positive set of $\phi_R$ is contained in $\{x_d> 0\}$, in particular $x_d >0$ on $\Gamma$.

Computing the derivative we find the formula
\begin{equation}\label{e.subsoln-integral}
\grad\phi_R(x)= e_d+(d-2)\int_{-R}^R \frac{x-y}{|x-y|^d} dy_d.
\end{equation}
We claim that the $e_d$ component of the integral above is positive on $\Gamma$. This is clear if $x_d \geq R$, while if $0<x_d <R$ we use symmetry to cancel parts of the integral, as in \fref{subsolution-barrier}, and thus we end up with
\[
\int_{-R}^{R} \frac{x_d-y_d}{|x-y|^d} dy_d = \int_{-R}^{2x_d-R} \frac{x_d-y_d}{|x-y|^d} dy_d \geq 0.
\]
It follows that $|\grad \phi_R| \geq |\partial_d \phi_R| \geq 1$ on $\Gamma$.

We explain the remainder of the proof heuristically, and then proceed with the rigorous details. We will look at the point $z_0 = z e_d \in \Gamma \cap \{|x'|=0\}$, which is the highest point on the free boundary and where we expect that the slope is most steep. We will use the implicit formula for the location of $z$ to show that $z \to Re_d$ as $R \to \infty$.  Using the integral formula for $\grad \phi_R(z_0)$ this will show that $|\grad \phi_R(z_0)| \to +\infty$ as $R \to +\infty$. Then we can use continuity of the gradient and hyperbolic rescaling to achieve a barrier with large gradient in a unit neighborhood of the origin.

Now we make this precise. Let $z_0:=ze_d\in\Gamma$. There is a unique such point since the previous arguments show that $z \mapsto \phi_R(ze_d)$ has derivative at least $1$ on $R < z < +\infty$ and approaches $-\infty$ as $z\searrow R$. Writing out the equation $\phi_R(z_0) = 0$
\[z = \int_{-R}^R (z-y_d)^{2-d} dy_d.\]
By computing the integral in the previous equation we arrive at the implicit formula
 \begin{equation}\label{computation}
z = \begin{cases}
    \log (1+\frac{2R}{z-R}) & d = 3\\
    (z-R)^{3-d}- (z+R)^{3-d} & d>3.
\end{cases} 
\end{equation}
We claim that $z \searrow R$ as $R \to \infty$. From \eqref{computation} it follows that 
\[(z-R) \lesssim
\begin{cases}
Re^{-R} & d = 3\\
R^{\frac{1}{d}-3} & d >3.
\end{cases}
\]

Now note the formula, following from \eref{subsoln-integral}
\begin{equation}\label{e.subsoln-integral2}
    \partial_d \phi_R(z_0) = e_d + c(d-2)\int_{-R}^R \frac{1}{|z_d-y_d|^{d-1}} \ dy_d.
\end{equation}
Then it follows from \eref{subsoln-integral2} and $z \searrow R$ that $|\grad\phi_R(z_0)| \to +\infty$ since $f(a) = |a|^{1-d}$ is not integrable near zero for $d\geq 2$.

Lastly, by choosing $R$ sufficiently large and performing hyperbolic rescaling by 
\[\phi(x):= a^{-1}\phi_R(ax +ze_d)\]
we can generate a subsolution with slope exceeding any particular desired threshold on $\Gamma \cap B_1(0)$.

\end{proof}

\subsubsection{Barriers for large defects: supersolutions}\label{s.barrier-supersolution}

To create a supersolution barrier we instead pull the half-plane solution outward via a positive source term in the complement of the positivity set.   Unlike the subsolution case before, this construction does work in all dimensions $d \geq 2$. Let $r>0$ be a free parameter and define
\[
\psi_r^0(x):= x_d + \Phi(x+(r+1)e_d)
\]
where $\Phi$ is the normalized Laplace fundamental solution in dimension $d$
\[\Phi(x) := \begin{cases} 
-\log |x| & d = 2\\
\frac{1}{d-2}\frac{1}{|x|^{d-2}} & d \geq 3.
\end{cases} \ \hbox{ with }  \ \grad \Phi(x) = -\frac{x}{|x|^d}.\]
 Notice that
\[\partial_{x_d}\psi_r^0(-re_d)=0\]
The heuristic idea is to choose $r$ so that $-re_d$ is (almost) on the free boundary and so the gradient will be very small near $-re_d$, then we can perform a hyperbolic rescaling centered at $-re_d$ to get small gradient in $B_1$. Note that in that case the zero level set of $\psi_r$ has a saddle-type singularity at $-re_d$ so we want to perturb slightly away from this scenario so that the free boundary is smooth.

We choose $r> \Phi(1)$ so that 
\[
\psi_r^0(-re_d) = -r + \Phi(1)  <0.
\]
For this range of $r$, along $x_d$-axis there are three free boundary points $x^i = s_ie_d$, $i=0,1,2$ with 
\[s_2<-(r+1)<s_1<-r< s_0<0.\]
See \fref{supersolution-construction}. The positive set $\{\psi_r^0>0\}$ then has two connected components, $\Omega_1$ that contains the half space $\{x_d\geq 0\}$ and has $x^0$ on its boundary, and a bounded component $\Omega_2$ containing $x^1$ and $x^2$ on its boundary. Define
\begin{equation}\label{e.psiR-defn}
    \psi_r(x):=\psi_r^0(x){\bf 1}_{\Omega_1}.
\end{equation}

\begin{figure}
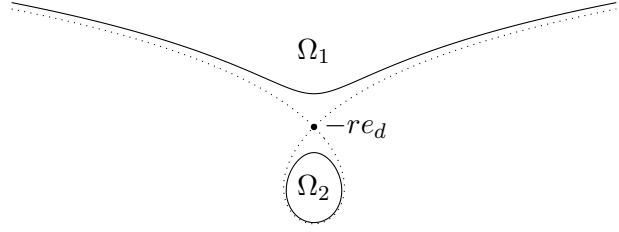

    \centering
    \begin{tikzpicture}
    \input{figures/contour_data.tikz}
    \input{figures/contour_data_sing.tikz}
    \node at (0,1) {$\Omega_1$};
    \node at (0,-.8){$\Omega_2$};
    \filldraw (0,0) circle (1pt) node[right]{$-re_d$};
\end{tikzpicture}
    \caption{Supersolution barrier construction. Zero level set of $\psi_r(x)$ for an $r>\Phi(1)$ plotted with solid line, the level set passing through $-re_d$ is also plotted with a dotted line. }
    \label{f.supersolution-construction}
\end{figure}

\begin{lemma}\label{supersolution} For any $\sigma>0$ and $d\geq 3$, there are $a>0$ and $r > \Phi(1)$ such that  ${\psi}(x):= a^{-1} \psi_r(ax+s_0e_d)$ is a supersolution of \eref{defect-prob}. 
\end{lemma}

\begin{proof}
    Let $s_0\in (-r,0)$ as above and $s_0e_d$ be the point on the intersection of $\partial \Omega_1  = \partial \{\psi_r>0\}$ with the $x_d$-axis. We claim that $s_0e_d$ is the lowest point on the free boundary, more precisely that 
  \begin{equation}\label{e.observation}
x_d \geq s_0 \ \hbox{ for all } \ x \in \partial \Omega_1.
\end{equation}
This is clear from \fref{supersolution-construction}, and follows from the fact that $\psi^0_r(x',x_d)$ is strictly monotonically decreasing as $|x'|$ increases for fixed $x_d$.

Now let us show the supersolution property $|D\psi_r|\leq 1$ on $\partial\Omega_1$.  Let $x \in \partial \Omega_1$. Denote $M(x):= |x+(r+1)e_d|$ and $A(x):= x_d+r+1$. Note that \eref{observation} implies that
\[A = x_d + r + 1 \geq s_0 + r + 1 \geq 1 \ \hbox{ and also } \ M = (A^2+|x'|^2)^{1/2} \geq A \geq 1.\]
Now we can compute the slope
\begin{align*}
|D\psi_r(x)|^2 &= \left|e_d - \tfrac{(x+(r+1)e_d)}{|x+(r+1)e_d|^d}\right|^2\\
&= 1-2 AM^{-d} + M^{-2(d-1)}\\ 
&\leq 1-2M^{-(d-1)}+ M^{-2(d-1)}\\
&= (1-M^{1-d})^2\\
&\leq1.
\end{align*}
This verifies the supersolution condition along the free boundary.

Note that as $r \searrow \Phi(1)$ then $|D\psi_r|(s_0e_d) \to 0$. By continuity of the gradient and hyperbolic rescaling $\psi(x) :=a^{-1}\psi_r(ax+s_0e_d)$ we can ensure that $\sup_{B_1}|\grad \psi|^2 \leq \sigma$. Note that, following from the definition of $\psi^0_r$, in $d \geq 3$
\[\psi(x)  = x_d+s_0+\frac{a^{1-d}}{d-2}\frac{1}{|x|^{d-2}} + O(|x|^{1-d}) \ \hbox{ as } \ \{\psi>0\} \ni x \to \infty,\]
and in $d=2$
\[\psi(x)  = x_d-a^{-1}\log|x| + O(1) \ \hbox{ as } \ \{\psi>0\} \ni x \to \infty.\]

\end{proof}

  \bibliographystyle{amsplain}
\bibliography{single-site-articles}

\end{document}